\documentclass[10pt,leqno]{amsart}
\usepackage{amsmath,amssymb,latexsym}
\usepackage{amscd}
\usepackage[OT4]{fontenc}
\usepackage{lmodern}
\usepackage{mathrsfs}
\usepackage{amsthm}
\usepackage{amsfonts}
\usepackage{url}
\usepackage[fleqn,tbtags]{mathtools}
\usepackage{enumerate}
\usepackage{graphicx}
\usepackage{color}
\usepackage{hyperref}
\usepackage[all]{xy}
\usepackage{tikz}
\usepackage{marginnote}
\usepackage[foot]{amsaddr}
\usepackage{orcidlink}
\usepackage{todonotes}

\newtheorem{prop}{Proposition}
\newtheorem{theorem}[prop]{Theorem}

\newtheorem{corollary}[prop]{Corollary}

\theoremstyle{definition}
\newtheorem{definition}[prop]{Definition}

\newtheorem{remark}[prop]{Remark}
\newtheorem{example}[prop]{Example}
\newtheorem*{Acknow*}{Acknowledgements}

\newcommand{\R}{\mathbb R}
\newcommand{\N}{\mathbb N}

\newcommand{\C}{\mathbb C}

\newcommand{\cI}{\mathcal I}

\newcommand{\im}{\operatorname{im}}

\newcommand{\supp}{\operatorname{supp}}

\let\epsilon\varepsilon
\let\phi\varphi
\let\rho\varrho

\title[Runge approximation results for spaces of smooth Whitney jets]{Runge type approximation results for spaces of smooth Whitney jets}

\author[T.~Cia\'s]{Tomasz Cia\'s$^1$\orcidlink{0000-0001-7119-7814}}

\address{$^1$Faculty of Mathematics and Computer Science, Adam Mickiewicz University, Pozna\'n, ul.~Uniwersytetu Pozna\'nskiego 4, 61-614 Pozna\'n, Poland}
\email{tomasz.cias@amu.edu.pl}
 
\author[T.~Kalmes]{Thomas Kalmes$^2$\orcidlink{0000-0001-7542-1334}}

\address{$^2$Faculty of Mathematics, Chemnitz University of Technology, 09107 Chemnitz, Germany}
\email{thomas.kalmes@math.tu-chemnitz.de}

\date{}

\thanks{This version of the article has been accepted for publication, after peer review (when applicable) but is not the Version of Record and does not reflect post-acceptance improvements, or any corrections. The Version of Record is available online at: \href{https://doi.org/10.1007/s10231-026-01655-7}{https://doi.org/10.1007/s10231-026-01655-7}}

\begin{document}
	
\begin{abstract}
	We prove Runge type approximation results for linear partial differential operators with constant coefficients on spaces of smooth Whitney jets. Among others, we characterize when for a constant coefficient linear partial differential operator $P(D)$ and for closed subsets $F_1\subset F_2$ of $\R^d$ the restrictions to $F_1$ of smooth Whitney jets $f$ on $F_2$ satisfying $P(D)f=0$ on $F_2$ are dense in the space of smooth Whitney jets on $F_1$ satisfying the same partial differential equation on $F_1$. For elliptic operators we give a geometric evaluation of this characterization. Additionally, for differential operators with a single characteristic direction, like parabolic operators, we give a sufficient geometric condition for the above density to hold. Under mild additional assumptions on $\partial F_1$ and for $F_2=\R^d$ this sufficient conditions is also necessary. As an application of our work, we characterize those open subsets $\Omega$ of the complex plane satisfying $\Omega=\operatorname{int}\overline{\Omega}$ for which the set of holomorphic polynomials are dense in $A^\infty(\Omega)$, under the additional hypothesis that $\overline{\Omega}$ satisfies the strong regularity condition. Furthermore, for the wave operator in one spatial variable, a simple sufficient geometric condition on $F_1, F_2\subset\R^2$ is given for the above density to hold. For the special case of $F_2=\R^2$ this sufficient condition is also necessary under mild additional hypotheses on $F_1$.\\
	
	
\end{abstract}

\maketitle

\section{Introduction}\label{section:introduction}

A well-known consequence of Runge's classical theorem on rational approximation asserts that for open subsets $X_1 \subset X_2$ of the complex plane $\C$, every holomorphic function defined on $X_1$ is the local uniform limit of (restrictions to $X_1$ of) holomorphic functions on $X_2$ if and only if $X_2$ does not contain a bounded connected component of $\C\setminus X_1$. During the 1950s, Lax \cite{Lax} and Malgrange \cite{Malgrange} independently generalized this result from holomorphic functions, i.e.~solutions of the homogeneous Cauchy-Riemann equation, to kernels of elliptic partial differential operators with constant coefficients. Further generalizations reached its peak with the work of Browder \cite{Browder1962b, Browder1962} in the 1960s dealing with elliptic partial differentiable operators with variable coefficients on various spaces of functions and distributions. Since then, so-called Runge approximation theorems have become a classical topic in the analysis of (linear) partial differerential operators. Recently, this subject has been rekindled due to its many applications, most prominenty in the deep work of Enciso, Peralta-Salas on the 3D Euler \cite{EnPe12,EnPe15} and (jointly with Luc\'a) the Navier-Stokes equations \cite{EnLuPe17}.

Much less is known, however, for non-elliptic partial differential operators. Malgrange \cite[Chapitre 1.2, Th\'eor\`eme 2]{Malgrange} (see also \cite[Theorem 10.5.1]{Hor2}) proved that for an arbitrary non-zero partial differential operator with constant coefficients, the linear span of the corresponding exponential solutions is dense in the kernel on the space of smooth functions over arbitrary convex open sets. Moreover, H\"ormander \cite[Theorem 10.5.2]{Hor2} and Tr\`eves \cite[Theorem 26.1]{Treves-LCS-PDE} established implicit approximation results for general constant coefficient partial differential operators applicable to arbitrary open sets $X_1, X_2$, and to $P$-convex open sets $X_1, X_2$, respectively. Nevertheless, apart from the heat operator \cite{Diaz1980,Jones1975} (see also \cite{GT}), for non-convex $X_1$, geometric conditions on $X_1$ and $X_2$ characterizing (or at least implying) the validity of Runge approximation theorems have only been established recently. For variable coefficients parabolic operators of second degree, this has been achieved by Enciso, Garc\'{\i}a-Ferrero, Peralta-Salas \cite{EnGaPe19} and has been complemented by Enciso, Peralta-Salas \cite{EnPe21} as well as Shlapunov, Vilkov \cite{ShVi24}, while for several classes of non-elliptic constant coefficient partial differential operators such results are due to the second author \cite{Kalmes2021} as well as to Debrouwere and the second author \cite{DebKalmes2024}.

While all above mentioned results deal with Runge approximation theorems for kernels of partial differential operators on spaces of functions or distributions defined on open subsets of $\R^d$, to the best of the authors' knowledge, there are no approximation results for kernels of constant coefficients partial differential operators on spaces of smooth Whitney jets over closed subsets of $\R^d$. It is the purpose of this article to contribute to the latter setting. 

We now state a sample of our results. For the definition of a $P$-Runge pair for smooth Whitney jets, see Definition \ref{definition:P-Runge pair} in Section \ref{section:preliminaries} below. Combining Corollary \ref{corollary:general result for hypoelliptic} with Theorem \ref{theorem:elliptic} gives the first theorem.\\

\textbf{Theorem A.} \textit{Let $F_1\subset F_2$ be closed subsets of $\R^d$ and let $P\in\C[\xi_1,\ldots,\xi_d]$ be non-constant. Consider the following conditions.}
\begin{itemize}
	\item[(i)] \textit{$(F_1, F_2)$ is a $P$-Runge pair for smooth Whitney jets.}
	\item[(ii)] \textit{$F_2$ does not contain a bounded connected component of $\R^d\setminus F_1$.} 
\end{itemize}
\textit{Then, (i) implies (ii). Additionally, if $P$ is elliptic, (ii) also implies (i).}\\

For the definition of $\deg_{\xi_1} Q_k$ in the next theorem, see Theorem \ref{thm} in Section \ref{section:non-elliptic} below. Moreover, for the definition of continuous boundary, see Proposition \ref{prop:continuous boundary for d=2} in Section \ref{section:non-elliptic} below. Additionally, we denote by $e_k=(\delta_{j,k})_{1\leq j\leq d}$ (Kronecker's $\delta$), $k=1,\ldots, d$, the canonical basis vectors in $\R^d$. We would like to point out that the hypotheses on $P$ in Theorem B are satisfied by polynomials of the form $P(\xi_1,\ldots,\xi_d)=c\xi_1^k+R(\xi_2,\ldots,\xi_d)$, where $R$ is an elliptic polynomial in $d-1$ variables of degree strictly larger than $k\in\N_0$ and $c\in\C$. In particular, this covers parabolic partial differential operators/polynomials with $\xi_1$ in the r\^{o}le of the time variable like the heat operator $P(D)=\partial_1-\sum_{j=2}^d\partial_j^2$, or the time-dependent free Schr\"odinger operator, again with $\xi_1$ in the r\^{o}le of the time variable, $P(D)=i\partial_1+\sum_{j=2}^d \partial_j^2$. \\

\textbf{Theorem B.}\label{thm B} \textit{Let $d\geq 2$ and let $P$ be a polynomial of degree $m$ with principal part $P_m$ such that $\{\xi\in\R^d: P_m(\xi)=0\}=\operatorname{span}\{e_1\}$. Moreover, assume that
	$$P(\xi_1,\ldots,\xi_d)=\sum_{k=0}^mQ_k(\xi_1,\ldots,\xi_{d-1})\xi_d^k,$$
	where $Q_k\in\C[\xi_1,\ldots,\xi_{d-1}]$ satisfies $\operatorname{deg}_{\xi_1}Q_k< m-k$ for every $0\leq k\leq m-1$. Additionally, let $F_1\subset F_2$ be closed subsets of $\R^d$.}

	\textit{Consider the following conditions.}
	\begin{itemize}
		\item[(i)] \textit{There is no $c\in\R$ such that $(\R^d\setminus F_1)\cap \{c\}\times\R^{d-1}$ has a bounded connected component contained in $F_2$.}
		\item[(ii)] \textit{$(F_1, F_2)$ is a $P$-Runge pair for smooth Whitney jets.}
		\item[(iii)] \textit{There are no $c\in\R, \varepsilon>0$ such that $(\R^d\setminus F_1)\cap (c-\varepsilon,c+\varepsilon)\times\R^{d-1}$ has a bounded connected component contained in $F_2$.}
	\end{itemize}
	\textit{Then, (i) implies (ii) and (ii) implies (iii). Furthermore, under each of the following additional hypotheses, the above conditions are equivalent.}
	\begin{itemize}
		\item[(a)] \textit{Let $d\geq 3$, $F_2=\R^d$, and let $F_1$ have $C^1$-boundary such that the normal space of every $\xi\in\partial F_1$ is not spanned by $e_1$.}
		\item[(b)] \textit{Let $d=2$, $F_2=\R^2$, and let $F_1$ have continuous boundary such that for every $\xi\in\partial F_1$ and every neighborhood $U$ in $\R^2$ of $\xi$ it holds}
		$$U\cap \partial F_1\cap\{x\in \R^2: x_1 <\xi_1\}\neq \emptyset \quad\textit{and}\quad U\cap \partial F_1\cap\{x\in \R^2: x_1 >\xi_1\}\neq \emptyset.$$
	\end{itemize}

The paper is organized as follows. In Section \ref{section:preliminaries} we recall some facts about smooth Whitney jets. In Section \ref{section:arbitrary operators} we give a complete characterization of those pairs of closed subsets $F_1\subset F_2$ of $\R^d$ which are a $P$-Runge pair for a given constant coefficient partial differential operator $P(D)$. For elliptic differential operators, we evaluate this condition in Section \ref{section:elliptic} to derive the geometric formulation given in Theorem A above. Additionally, we apply our results to certain subspaces of holomorphic/harmonic functions (Examples \ref{example:holomorphic} and \ref{example:harmonic}), thereby settling the open problem \cite[Introduction]{GaNe16} when holomorphic polynomials in one complex variable are dense in $A^\infty(\Omega)$ for open subsets $\Omega$ of the complex plane for which $\Omega=\operatorname{int}\overline{\Omega}$ and $\overline{\Omega}$ satisfies the strong regularity condition (see Corollary \ref{cor:A infinity}). In the final Section \ref{section:non-elliptic} we evaluate the characterization of $P$-Runge pairs for smooth Whitney jets for certain classes of non-elliptic differential operators which contain parabolic operators. Additionally, the proof of Theorem B is given in Section \ref{section:non-elliptic}.

\section{Preliminaries}\label{section:preliminaries}

In this short preliminary section we discuss smooth Whitney jets and properties of partial differential operators acting on spaces of smooth Whitney jets. The purpose of this section is mainly to introduce the notation and some essential notions we shall work with in this article. For anything related to functional analysis which is not explained in the text, we refer the reader to \cite{MeV}, and for notions and results about distributions and partial differential operators we refer to \cite{Hor1, Hor2}. We emphasize that by an isomorphism between locally convex spaces we mean a bijective continuous linear mapping with a continuous inverse.

As is customary, we denote the euclidean norm of a vector $x$ in $\R^k$ by $|x|$. While we use the same notation for all $k$, the dimension will always be clear from the context. In addition, for a multi-index $\alpha\in\N_0^k$ we denote its length by $|\alpha|(=\sum_{j=1}^k \alpha_j)$ as well. Again, this should not cause any confusion since the meaning will be apparent from the context. Moreover, for $x\in \R^k$ and $\varepsilon>0$ we write $B(x,\varepsilon)$ for the open euclidean ball about $x$ with radius $\varepsilon$, where we again suppress the concrete dimension $k$ in the notation, as this will be clear from the context.

For an open subset $\Omega$ of $\R^d$ we denote by $\mathscr{E}(\Omega)$ the Fr\'echet space of smooth, complex valued functions on $\Omega$ which is endowed with the sequence of seminorms $(|\cdot|_n)_{n\in\N}$ defined by
\begin{equation}\label{eq:seminorms on E}
	\forall\,g\in\mathscr{E}(\Omega)\colon\,|g|_n:=\sup\{|\partial^{\alpha}g(x)|:x\in K_n, |\alpha|\leq n\},
\end{equation}
where $(K_n)_{n\in\N}$ is an arbitrary compact exhaustion of $\Omega$. Let $\mathscr{E}'(\Omega)$ be its dual space equipped -- as other duals in this paper -- with the strong topology. As is well-known, $\mathscr{E}'(\Omega)$ is the space of distributions on $\R^d$ whose support is a compact subset of $\Omega$. For $\Omega=\R^d$ we abbreviate, as usual, $\mathscr{E}=\mathscr{E}(\R^d)$ and $\mathscr{E}'=\mathscr{E}'(\R^d)$.

For a compact subset $K$ of $\R^d$ we denote by $\mathscr{E}(K)$ the space of smooth Whitney jets on $K$, that is, families $f=(f^{(\alpha)})_{\alpha\in\N_0^d}$ of continuous (real or complex valued) functions $f^{(\alpha)}$ on $K$ for which, with the formal Taylor polynomials of arbitrary order $n$
$$T^n_xf(y)=\sum_{|\alpha|\leq n}\frac{f^{(\alpha)}(x)}{\alpha!}(y-x)^\alpha,$$
the quantities
$$q_{K,n,t}(f):=\sup\left\{\frac{|f^{(\alpha)}(y)-\partial^\alpha T_x^nf(y)|}{|x-y|^{n-|\alpha|}}\colon |\alpha|\leq n, x,y\in K, 0<|x-y|<t\right\}$$
satisfy $\lim_{t\rightarrow 0}q_{K,n,t}(f)=0$ for all $n\in\N_0$. (By the very definition, smooth Whitney jets are not usual functions, despite the fact that they are sometimes called Whitney functions in the literature.) On $\mathscr{E}(K)$ we define the sequence $(\|\cdot\|_{n,K})_{n\in\N_0}$ of seminorms by
$$\forall\,f\in\mathscr{E}(K)\colon \|f\|_{n,K}:=\sup\{|f^{(\alpha)}(x)|\colon x\in K, |\alpha|\leq n\}+\sup\{q_{K,n,t}(f)\colon t>0\}$$
which makes $\mathscr{E}(K)$ into a Fr\'echet space.

For a closed subset $F$ of $\R^d$, let $(K_l)_{l\in\N}$ be a compact exhaustion of $F$. Then, for $l\in\N$ the linear mapping
$$\rho_{l+1}^l:\mathscr{E}(K_{l+1})\rightarrow\mathscr{E}(K_l), f=(f^{(\alpha)})_{\alpha\in\N_0}\mapsto f_{\mid K_l}:=(f^{(\alpha)}_{\mid K_l})_{\alpha\in\N_0^d}$$
is correctly defined and continuous.
As usual, $\mathscr{E}(F)$ is defined as the locally convex projective limit of $(\mathscr{E}(K_l))_{l\in\N}$, i.e.\ 
$$\mathscr{E}(F):=\operatorname{proj}_{\leftarrow l}\mathscr{E}(K_l)=\left\{(f_l)_{l\in\N}\in\prod_{l\in\N}\mathscr{E}(K_l)\colon \rho_{l+1}^l(f_{l+1})=f_l,\quad l\in\N\right\}$$
equipped with the subspace topology of the topological product $\prod_{l\in\N}\mathscr{E}(K_l)$. A moment's reflection shows that $\mathscr{E}(F)$ does not depend on the particular choice of $(K_l)_{l\in\N}$.

Clearly, Taylor's Theorem implies that for $g\in\mathscr{E}$ it holds $(\partial^\alpha g_{\mid F})_{\alpha\in\N_0^d}$ (or, more precisely, $\left(\left(\partial^{\alpha}g_{\mid K_l}\right)_{\alpha\in\N_0^d}\right)_{l\in\N}$) belongs to $\mathscr{E}(F)$. A seminal result of Whitney \cite{Whitney34} (see also \cite{Malgrange1967}) says that every $f\in\mathscr{E}(F)$ has such an extension $g$. Whitney's methods yield that $\mathscr{E}(F)$ is a Fr\'echet space and hence, by the Closed Graph Theorem, the linear mapping $$\rho_F\colon\mathscr{E}\rightarrow\mathscr{E}(F), g\mapsto (\partial^\alpha g_{\mid F})_{\alpha\in\N_0^d}$$
is continuous. Obviously, 
$$\cI_F=\{g\in\mathscr{E}\colon \partial^\alpha g_{\mid{F}}=0 \text{ for all }\alpha\in\N_0^d\}$$
is the kernel of $\rho_F$. By Whitney's result,
$$0\rightarrow\cI_F\hookrightarrow\mathscr{E}\xrightarrow{\rho_F}\mathscr{E}(F)\rightarrow 0$$
is a short exact sequence of Fr\'echet spaces. It follows herefrom and the Open Mapping Theorem that $\mathscr{E}(F)$ is isomorphic to the quotient space $\mathscr{E}/\cI_F$. In particular, the sequence of seminorms $(\|\cdot\|_n)_{n\in\N}$ on $\mathscr{E}(F)$ defined as
$$\forall\,f\in\mathscr{E}(F)\colon\,||f||_n:=\inf\left\{|g|_{n}\colon g\in\mathscr{E}  \text{ and } \partial^\alpha g_{\mid{F}}=f^{(\alpha)}, \alpha\in\N_0^d\right\}$$
also defines the Fr\'echet space topology on $\mathscr{E}(F)$.

We also define
$$i_F\colon\mathscr{E}/\cI_F\to\mathscr{E}(F), g+\cI_F\mapsto (\partial^\alpha g_{\mid F})_{\alpha\in\N_0^d}$$
which by the above is easily seen to be a well-defined isomorphism. With $q_F:\mathscr{E}\rightarrow\mathscr{E}/\cI_F$ denoting the quotient map, it follows
$$\rho_F=i_F\circ q_F\colon\mathscr{E}\to\mathscr{E}(F), f\mapsto(\partial^\alpha f_{\mid F})_{\alpha\in\N_0^d}.$$
We denote the (strong) dual of $\mathscr{E}(F)$ by $\mathscr{E}'(F)$.

Defining for $f=(f^{(\alpha)})_{\alpha\in\N_0^d}\in\mathscr{E}(F)$
$$D_jf=-i(f^{(\alpha+e_j)})_{\alpha\in\N_0^d}, \quad j=1,\ldots, d,$$
where $e_j=(0,\ldots,0,1,0,\ldots,0)$ is the $j$-th unit vector it follows immediately that
$$\forall\,g\in\mathscr{E}\colon D_j(\rho_F g)=\rho_F\left(-i\partial_j g\right).$$
More generally, for a polynomial $P\in\C[\xi_1,\ldots,\xi_d]$ of degree $m$, $P(\xi)=\sum_{|\alpha|\leq m}c_\alpha \xi^\alpha$, and the corresponding partial differential operator of order $m$ with constant coefficients $P(D)$ on $\mathscr{E}$,
$$P(D)\colon\mathscr{E}\rightarrow\mathscr{E}, g\mapsto\sum_{|\alpha|\leq m}c_\alpha (-i)^{|\alpha|}\partial^\alpha g$$
and
$$P_F(D)\colon \mathscr{E}(F)\rightarrow\mathscr{E}(F), f\mapsto \sum_{|\alpha|\leq m}c_\alpha D_1^{\alpha_1}\cdots D_d^{\alpha_d}f$$
we have $P_F(D)\circ\rho_F=\rho_F\circ P(D)$ and $P_F(D)$ is a continuous linear operator. By a slight abuse of notation, instead of $P_F(D)$ we sometimes also write $P(D)$. As observed by Frerick \cite[Introduction to Chapter 6]{Frerick-Habilitation}, this equality combined with the surjectivity of $P(D)\colon\mathscr{E} \to\mathscr{E} $ easily implies the surjectivity of $P_F(D)$. This stands in contrast to $P(D)$ acting on the space of smooth functions $\mathscr{E}(\Omega)$ on an open subset $\Omega$ of $\R^d$ where it is not surjective in general.

Let $\mathscr{E}_P(F)$ denote the kernel of $P(D)\colon\mathscr{E}(F)\to\mathscr{E}(F)$, i.e.
$$\mathscr{E}_P(F)=\{f\in\mathscr{E}(F):\;P(D)f=0\}.$$

\begin{definition}\label{definition:P-Runge pair}
	For a pair $(F_1,F_2)$ of closed subsets $F_1\subset F_2$ of $\R^d$ we define the (continuous) restriction operator
	\begin{equation}\label{eq-restriction-operator}
		r\colon\mathscr{E}(F_2)\to\mathscr{E}(F_1), f\mapsto\left({f^{(\alpha)}}_{\mid F_1}\right)_{\alpha\in\N_0^d}.
	\end{equation}
	Clearly, $r\left(\mathscr{E}_P(F_2)\right)\subset\mathscr{E}_P(F_1)$, and we call $(F_1, F_2)$ a \emph{$P$-Runge pair (for smooth Whitney jets)} if $r\left(\mathscr{E}_P(F_2)\right)$ is dense in $\mathscr{E}_P(F_1)$.
\end{definition}

In order to avoid unnecessary lengthy formulations, for obvious reasons we assume from now on that all polynomials are non-constant and that for every pair $(F_1,F_2)$ of closed subsets $F_1\subset F_2$ of $\R^d$ it holds $F_1\neq F_2$.

Recall that for a polynomial $P\in\C[\xi_1,\ldots,\xi_d]$ of degree $m$, $P(\xi)=\sum_{|\alpha|\leq m}c_\alpha \xi^\alpha$, its principal part is defined as $P_m(\xi)=\sum_{|\alpha|= m}c_\alpha \xi^\alpha$. Moreover, recall that $P$ (and the differential operator $P(D)$) are called elliptic if $P_m(\xi)\neq 0$ for every $\xi\in\R^d\backslash\{0\}$ while $P$ (and the differential operator $P(D)$) are hypoelliptic if for every open subset $X$ of $\R^d$ and each $u\in\mathscr{D}'(X)$ with $P(D)u\in C^\infty(X)$ we have $u\in C^\infty(X)$. For a characterization of hypoelliptic polynomials, see \cite[Section 11.1]{Hor2}. Finally, let $H=\{x\in\R^d\colon \sum_{j=1}^d x_jN_j=c\}$ be a hyperplane in $\R^d$, where $N=(N_1,\ldots,N_d)\in\R^d\backslash\{0\}$, $c\in \R$. Then, $H$ is characteristic for $P$ (and $P(D)$) if $P_m(N)=0$.

\section{Runge pairs for smooth Whitney jets}\label{section:arbitrary operators}

In this section we give a complete characterization of Runge pairs for arbitrary partial differential operators in the context of smooth Whitney jets. In order to do so, we first provide the following abstract result. As usual, for a continuous linear operator $A\colon E_1\to E_2$ between locally convex spaces, we denote its transpose by
$\prescript{t}{}A\colon E_2'\to E_1'$. Then, $\prescript{t}{}A$ is continuous when the duals $E_1'$ and $E_2'$ are both equipped with the strong topology (see \cite[Proposition 23.30]{MeV}). Moreover, by a standard consequence of the Hahn-Banach Theorem, $\im A$ is dense in $E_2$ if and only if $\prescript{t}{}A$ is injective (see \cite[Lemma 23.31]{MeV}).

From the definition of the spaces of smooth Whitney jets it follows immediately that the operator $r$ in \eqref{eq-restriction-operator} is surjective. Consequently,
$\prescript{t}{}{r}\colon \mathscr{E}'(F_1)\to \mathscr{E}'(F_2)$
is an injective continuous linear operator.

\begin{prop}\label{abstract-criterion}
	Let $E_1,E_2$ be Fr\'echet spaces and let $T_1\colon E_1\to E_1$, $T_2\colon E_2\to E_2$ and $r\colon E_2\to E_1$ be continuous linear maps. Assume moreover that $T_2$ is surjective and $r\circ T_2=T_1\circ r$. Then the following assertions are equivalent:
	\begin{enumerate}[{\upshape(i)}]
		\item $T_1$ is surjective and $r(\ker T_2)$ is dense in $\ker T_1$;
		\item $r$ has dense range and for all $y\in E_2'$ it holds $y\in\im{}^t r$ whenever ${}^tT_2(y)\in\im{}^t r$.
	\end{enumerate}
\end{prop}

\begin{proof}
	This is an immediate consequence of \cite[Proposition 13.1 and Corollary 3 on p.~54]{Treves-LCS-PDE}. 
\end{proof}

\begin{remark}\label{remark:motivation}
	Applying the previous Proposition to $E_j=\mathscr{E}(F_j)$ and to $T_j=P(D)$, $j=1,2$, it follows from the considerations proceeding Definition \ref{definition:P-Runge pair} that $(F_1, F_2)$ is a $P$-Runge pair precisely when $u\in\im\prescript{t}{}{r}$ for every $u\in\mathscr{E}'(F_2)$ with $\prescript{t}{}{P(D)}u\in\im\prescript{t}{}{r}$.
	
	Hence, for a closed subset $F$ of $\R^d$, we next collect some facts about $\mathscr{E}'(F)$ and $\prescript{t}{}{P(D)}$. With the annihilator/polar $\cI_F^\circ$ of $\cI_F$, i.e.
	$$\cI_F^\circ:=\left\{u\in\mathscr{E}'\colon u(f)=0 \text{  for all } f\in\cI_F\right\}$$
	which is obviously a closed subspace of $\mathscr{E}'$, the following holds.
\end{remark}

\begin{prop}\label{prop-IFperp}
	For every closed subset $F$ of $\R^d$, $\prescript{t}{}\rho_F:\mathscr{E}'(F)\to\cI_F^\circ$
	is a correctly defined isomorphism and $\cI_F^\circ$ consists precisely of those distributions on $\R^d$ which have compact support in $F$. Consequently, $\mathscr{E}'(F)$ can be identified with the distributions on $\R^d$ whose support is a compact subset of $F$. 
\end{prop}

\begin{proof}
	Since $i_F\colon\mathscr{E} /\cI_F\to\mathscr{E}(F), g+\cI_F\mapsto (\partial^\alpha g_{\mid F})_{\alpha\in\N_0^d}$ is an isomorphism, the same holds for
	$$\prescript{t}{}i_F\colon\mathscr{E}'(F)\to\left(\mathscr{E} /\cI_F\right)'.$$
	Since $\mathscr{E} $ is a Fr\'echet-Schwartz space, by \cite[Corollary 26.25 and Remark 26.5]{MeV},
	$$\prescript{t}{}q_F:\left(\mathscr{E} /\cI_F\right)'\to\cI_F^\circ$$
	is a well-defined isomorphism which proves that $\prescript{t}{}q_F\circ\prescript{t}{}i_F=\prescript{t}{}(i_F\circ q_F)=\prescript{t}{}\rho_F$ is an isomorphism from $\mathscr{E}'(F)$ onto $\cI_F^\circ$.
	
	Obviously, the subspace $\{\varphi\in\mathscr{D}(\R^d)\colon \supp\varphi\cap F=\emptyset\}$ of $\mathscr{D}(\R^d)$ is contained in $\cI_F$. Therefore, we have
	$$\cI_F^\circ\subset \{\varphi\in\mathscr{D}(\R^d)\colon \supp\varphi\cap F=\emptyset\}^\circ,$$
	where the polar on the right hand side refers to the dual pair $\mathscr{E} $, $\mathscr{E}'$. Thus, one easily deduces from the definition of support of a distribution
	$$\{\varphi\in\mathscr{D}(\R^d)\colon \supp\varphi\cap F=\emptyset\}^\circ=\{u\in\mathscr{E}'\colon \supp u\subset F\}.$$
	On the other hand, let $u\in\mathscr{E}'$ have support in $F$. Then, $u(f)=0$ for every $f\in\cI_F$ by \cite[Theorem 2.3.3]{Hor1} so that $u\in\cI_F^\circ$. Summarizing,
	$$\cI_F^\circ=\{u\in\mathscr{E}'\colon \supp u\subset F\}$$
	which completes the proof.
\end{proof}

Combining Propositions \ref{abstract-criterion} and \ref{prop-IFperp} we can now derive the main result of this section.

\begin{theorem}\label{criterion-lax-malgrange}
	Let $F_1\subset F_2$ be closed subsets of $\R^d$ and let $P(D)$ be a partial differential operator with constant coefficients. Then the following assertions are equivalent.
	\begin{enumerate}[{\upshape(i)}]
		\item $(F_1, F_2)$ is a $P$-Runge pair.
		\item Every compactly supported distribution $u$ on $\R^d$ with support contained in $F_2$ has its support already contained in $F_1$ whenever $\check P(D)u$ is supported in $F_1$. Here, as usual, $\check{P}(\xi)=P(-\xi)$.
	\end{enumerate}
\end{theorem}

\begin{proof}
	For the proof, we distinguish again notationally and write $P_{F_j}(D):\mathscr{E}(F_j)\to\mathscr{E}(F_j)$, $j=1,2,$
	as well as $P(D):\mathscr{E} \rightarrow\mathscr{E}$.
	We then have $\rho_{F_j}\circ P(D)=P_{F_j}(D)\circ\rho_{F_j}$ for $j=1,2$, where
	$$\rho_{F_j}:\mathscr{E} \rightarrow\mathscr{E}(F_j),f\mapsto (\partial^\alpha f_{\mid F_j})_{\alpha\in\N_0^d}.$$
	As already observed in Remark \ref{remark:motivation}, by Proposition \ref{abstract-criterion} applied to $E_j=\mathscr{E}(F_j)$, $T_j=P_{F_j}(D)$
	$j=1,2$, and the restriction operator $r\colon\mathscr{E}(F_2)\to\mathscr{E}(F_1)$, $(F_1,F_2)$ is a $P$-Runge pair precisely when
	\begin{equation}\label{cond-1}
		\forall\,u\in\mathscr{E}'(F_2)\colon\left(\prescript{t}{}{P_{F_2}(D)}u\in\im\prescript{t}{}{r}\Rightarrow u\in\im\prescript{t}{}{r}\right).
	\end{equation}
	By Proposition \ref{prop-IFperp}, $\prescript{t}{}{\rho_{F_j}}:\mathscr{E}'(F_j)\to\cI_{F_j}^\circ$, $j=1,2$, are isomorphisms so that \eqref{cond-1} is equivalent to
	\begin{equation}\label{cond0}
		\forall\,v\in\cI_{F_2}^\circ\colon\left(\prescript{t}{}{P_{F_2}(D)}\left(\left(\prescript{t}{}{\rho_{F_2}}\right)^{-1}v\right)\in\im\prescript{t}{}{r}\Rightarrow \left(\prescript{t}{}{\rho_{F_2}}\right)^{-1}v\in\im\prescript{t}{}{r}\right).
	\end{equation}
	Clearly, $\rho_{F_1}=r\circ\rho_{F_2}$, and thus, $\prescript{t}{}{\rho_{F_1}}=\prescript{t}{}{\rho_{F_2}}\circ\prescript{t}{}{r}$ with injective $\prescript{t}{}{\rho_{F_2}}$. Due to the injectivity of $\prescript{t}{}{\rho_{F_2}}$, for $u\in \mathscr{E}'(F_2)$ it holds $u\in\im\prescript{t}{}{r}$ if and only if $\prescript{t}{}{\rho_{F_2}}u\in \im \prescript{t}{}{\rho_{F_1}}=\cI_{F_1}^\circ$. Therefore, \eqref{cond0} is equivalent to 
	\begin{equation}\label{cond1}
		\forall\,v\in\cI_{F_2}^\circ\colon\left(\prescript{t}{}{\rho_{F_2}}\left(\prescript{t}{}{P_{F_2}(D)}\left(\left(\prescript{t}{}{\rho_{F_2}}\right)^{-1}v\right)\right)\in\cI_{F_1}^\circ \Rightarrow v\in\im\prescript{t}{}{r}\right).
	\end{equation}
	Since $\prescript{t}{}{\rho_{F_2}}\circ\prescript{t}{}{P_{F_2}(D)}=\prescript{t}{}{P(D)}\circ \prescript{t}{}{\rho_{F_2}}$ and $\check{P}(D):\mathscr{E}'\rightarrow\mathscr{E}'$ is the transpose $\prescript{t}{}{P(D)}$ of $P(D)$ we have
	$$\forall\,v\in\cI_{F_2}^\circ\colon\prescript{t}{}{\rho_{F_2}}\left(\prescript{t}{}{P_{F_2}(D)}\left((\prescript{t}{}{\rho_{F_2}})^{-1}v\right)\right) =\check{P}(D)v$$
	so that \eqref{cond1} simplifies to
	\begin{equation}\label{cond2}
		\forall v\in \cI_{F_2}^\circ\colon\left(\check{P}(D)v\in\cI_{F_1}^\circ\Rightarrow v\in\cI_{F_1}^\circ\right).
	\end{equation}
	One final glimpse at Proposition \ref{prop-IFperp} reveals that \eqref{cond2} is equivalent to assertion (ii) which completes the proof.
\end{proof}

As an immediate consequence of Theorem \ref{criterion-lax-malgrange} we obtain the next result.

\begin{corollary}\label{corollary:intersections}
	Let $F^j_1\subset F^j_2$ be closed subsets of $\R^d$ where $(F_1^j,F_2^j)$ are $P$-Runge pairs for every $j$ from an arbitrary index set $I$. Then, $\left(\cap_{j\in I}F_1^j, \cap_{j\in I}F_2^j\right)$ is again a $P$-Runge pair.
\end{corollary}

\begin{corollary}\label{corollary:general result for hypoelliptic}
	Let $F_1\subset F_2$ be closed subsets of $\R^d$ such that $\R^d\setminus F_2$ is dense in $\R^d\setminus F_1$. Then, $(F_1, F_2)$ is a $P$-Runge pair for every hypoelliptic polynomial $P$.
\end{corollary}

\begin{proof}
	Fix a hypoelliptic polynomial $P$. Let $u\in\mathscr{E}'$ be such that $\supp u\subset F_2$ and $\supp\check{P}(D)u\subset F_1$. By \cite[Theorem 11.1.3]{Hor2} $\check{P}$ is hypoelliptic. Therefore, $u_{\mid\R^d\setminus F_1}\in C^\infty(\R^d\setminus F_1)$ because $\check{P}(D)u=0$ in $\R^d\setminus F_1$. Additionally, the smooth function $u_{\mid\R^d\setminus F_1}$ vanishes on $\R^d\backslash F_2$, so that by the density of the latter in $\R^d\backslash F_1$ we conclude $u_{\mid\R^d\setminus F_1}=0$ implying $\supp u\subset F_1$. It follows from Theorem \ref{criterion-lax-malgrange} that $(F_1, F_2)$ is a $P$-Runge pair.
\end{proof}

\begin{corollary}\label{corollary:compositions}
	Let $P_1,P_2$ be polynomials. Moreover, let $F_1\subset F_2$ be closed subsets of $\R^d$. Then, the following are equivalent.
	\begin{itemize}
		\item[(i)] $(F_1, F_2)$ is a Runge pair for both polynomials $P_1$ and $P_2$.
		\item[(ii)] $(F_1, F_2)$ is a $P_1 P_2$-Runge pair. 
	\end{itemize}
\end{corollary}

\begin{proof}
	That (i) implies (ii) follows immediately from Theorem \ref{criterion-lax-malgrange}. Similarly, if (ii) holds and $u\in\mathscr{E}'$ has support in $F_2$ and satisfies $\supp \check{P}_2(D)u\subset F_1$, it follows from $\supp \check{P}_1(D)\check{P}_2(D)u\subset F_1$, (ii) and Theorem \ref{criterion-lax-malgrange} that $\supp u\in F_1$. Therefore, $(F_1, F_2)$ is a $P_2$-Runge pair. That $(F_1, F_2)$ is also a $P_1$-Runge pair follows similarly taking into account that $\check{P}_1(D)\check{P}_2(D)=\check{P}_2(D)\check{P}_1(D)$.
\end{proof}

Next we prove a necessary condition for pairs of closed sets $(F_1, F_2)$ to be a $P$-Runge pair. In Section \ref{section:elliptic} we will show that this necessary condition is also sufficient for elliptic partial differential operators.

\begin{corollary}\label{cor:necessary condition}
	Let the pair of closed subsets $(F_1, F_2)$ be a $P$-Runge pair for some $P$. Then, $F_2$ does not contain a bounded connected component of $\R^d\backslash F_1$. 
\end{corollary}

\begin{proof}
	Since $P$ is non-constant there is $\zeta\in\C^d$ with $\check{P}(\zeta)=0$. Thus, $\check{P}(D)g_0=0$ for the real analytic function $g_0(x)=\exp\left(i\sum_{j=1}^d \zeta_j x_j\right)$, $x\in\R^d$.
	
	Assume that $G$ is a bounded (open) connected component of $\R^d\backslash F_1$ which is contained in $F_2$. Then,
	$$g:\R^d\to\C,x\mapsto\begin{cases}
		g_0(x),&\text{if }x\in G\\ 0,&\text{otherwise}
	\end{cases}$$
	is an integrable function defining the compactly supported distribution
	$$u_g:\mathscr{D}(\R^d)\to\C,\varphi\mapsto \int_{\R^d}g(x)\varphi(x)dx,$$
	and since $F_2$ is closed, $\supp u_g=\overline{G}\subset F_2$. In particular, $\supp u_g$ is not contained in $F_1$. On the other hand, $\check{P}(D) u_g$ vanishes on $G\cup (\R^d\backslash\overline{G})$, so
	$$\supp\check{P}(D)u_g\subset\partial G\subset F_1$$
	since $F_1$ is closed. By Theorem \ref{criterion-lax-malgrange}, $(F_1,F_2)$ is not a $P$-Runge pair.
\end{proof}

We close this section by considering the special case when our closed sets are closures of open subsets with nice boundary.\\

Let $\Omega$ be a non-empty open subset of $\R^d$. Then the elements $(f^{(\alpha)})_{\alpha\in\N_0^d}$ of $\mathscr{E}(\overline{\Omega})$ are uniquely determined by the function $f^{(0)}$ on $\Omega$, hence $\mathscr{E}(\overline\Omega)$ and $\mathscr{E}_P(\overline\Omega)$ can be treated as function spaces on $\Omega$. Also in this case, the kernel of a differential operator $P(D)$ on the space $\mathscr{E}(\overline\Omega)$ is determined by the kernel of the same operator on the space $C^\infty(\Omega)$ as we show in the next proposition. We equip the subspace $\mathscr{E}_{\overline{\Omega}}:=\left\{g\in C(\overline{\Omega})\colon g=\tilde{g}_{\mid\overline{\Omega}} \text{ for some }\tilde{g}\in\mathscr{E} \right\}$ of $C(\overline{\Omega})$ with the sequence of seminorms $(\widetilde{\|\cdot\|}_n)_{n\in\N}$ defined by 
$$\forall g\in\mathscr{E}_{\overline{\Omega}}\colon\widetilde{\|g\|}_n:=\inf\{|\tilde{g}|_n\colon \tilde{g}\in\mathscr{E}, \tilde{g}_{\mid\overline{\Omega}}=g\},$$
see \eqref{eq:seminorms on E}.

\begin{prop}\label{prop-kernels}
	Let $\Omega$ be a non-empty, open subset of $\R^d$. Then, $\mathscr{E}_{\overline{\Omega}}$ is a Fr\'echet space and
	$$\iota\colon\mathscr{E}(\overline{\Omega})\to \mathscr{E}_{\overline{\Omega}}, f=(f^{(\alpha)})_{\alpha\in\N_0^d}\mapsto f^{(0)}$$
	is a linear isomorphism. Moreover, it holds
	$$\iota(\mathscr{E}_P(\overline{\Omega}))=\left\{g_{\mid\overline{\Omega}}\colon g\in\mathscr{E} , P(D)g=0\text{ in }\Omega\right\}.$$
\end{prop}	

\begin{proof}
	It is easily seen that $$j_{\overline{\Omega}}\colon \mathscr{E}/\cI_{\overline{\Omega}}\to\mathscr{E}_{\overline{\Omega}}, g+\cI_{\overline{\Omega}}\mapsto g_{\mid\overline{\Omega}}$$
	is a well-defined isomorphism. With the canonical isomorphism $i_{\overline{\Omega}}\colon\mathscr{E}/\cI_{\overline{\Omega}}\to\mathscr{E}(\overline{\Omega})$ it follows $\iota=j_{\overline{\Omega}}\circ\iota_{\overline{\Omega}}^{-1}$ which proves that $\iota$ is an isomorphism. This also implies that $\mathscr{E}_{\overline{\Omega}}$ is a Fr\'echet space.
	
	Next, we observe that for $f\in\mathscr{E}_P(\overline{\Omega})$ trivially $\iota(f)\in\{g_{\mid\overline{\Omega}}\colon g\in\mathscr{E} , P(D)g=0\text{ in }\Omega\}$ due to $P_{\overline{\Omega}}(D)\circ\rho_{\overline{\Omega}}=\rho_{\overline{\Omega}}\circ P(D)$, where $P(D)\colon\mathscr{E} \to\mathscr{E} $. On the other hand, for $g\in\mathscr{E} $ with $P(D)g=0$ in $\Omega$ we conclude from the density of $\Omega$ in $\overline{\Omega}$ that $P(D)g\in\cI_{\overline{\Omega}}^\circ=\ker\rho_{\overline{\Omega}}$. Applying again $P_{\overline{\Omega}}(D)\circ\rho_{\overline{\Omega}}=\rho_{\overline{\Omega}}\circ P(D)$ it holds $\rho_{\overline{\Omega}}(g)\in\mathscr{E}_P(\overline{\Omega})$. From $\iota(\rho_{\overline{\Omega}}g)=g_{\overline{\Omega}}$ it follows indeed $\iota(\mathscr{E}_P(\overline{\Omega}))=\left\{g_{\mid\overline{\Omega}}\colon g\in\mathscr{E} , P(D)g=0\text{ in }\Omega\right\}.$
\end{proof}

By Proposition \ref{prop-kernels}, $\mathscr{E}_P(\overline{\Omega})$ can be identified with the Fr\'echet space of solutions in $C^\infty(\Omega)$ of the equation $P(D)f=0$ admitting a smooth extension to $\R^d$. In other words, $\mathscr{E}_P(\overline{\Omega})$ is the space of functions on $\overline\Omega$ admitting a smooth extension to $\R^d$ and satisfying the equation $P(D)f=0$ on $\Omega$.\\

If the boundary of $\Omega$ is sufficiently regular -- for example, if $\overline{\Omega}$ satisfies the so-called strong regularity condition, see below -- then some natural definitions of smooth functions on closed sets coincide (see \cite[Lemma 1.10]{Rainer2019} and \cite[Proposition 2.16]{Bierstone1980}).
Let us recall (see e.g. \cite[$(2.16.1)$]{Bierstone1980}) that a path connected closed set $F\subset\R^d$ with $F=\overline{\operatorname{int}F}$ satisfies \emph{the strong regurality condition} if
for all $x_0\in F$ there are $\epsilon>0$, $C>0$ and $\theta\in(0,1]$ such that any two points $x,y\in F\cap B(x_0,\epsilon)$ can be joined by a rectifiable arc $\gamma\subset F$ with length $|\gamma|$ such that 
\begin{enumerate}
	\item $\gamma$ meets the boundary of $F$ at most finitely many times
	\item $|\gamma|\leq C|x-y|^{\theta}$.
\end{enumerate}
It can be shown that $\overline{\Omega}$ satisfies the strong regularity condition whenever $\Omega$ has a H\"older continuous boundary. As we could not find any reference for this, for the reader's convenience, a proof of this fact, together with the definition of a H\"older continuous boundary, is given in the appendix.

By \cite[Proposition 2.16]{Bierstone1980}, in case $\Omega$ is such that $\overline{\Omega}$ satisfies the strong regularity condition, it holds $\mathscr{E}(\overline{\Omega})=C^\infty(\overline{\Omega})$, where for an arbitrary open subset $U$ of $\R^d$ 
$$C^\infty(\overline{U}):=\{f\in C^\infty(U)\colon \partial^\alpha f\text{ admits a continuous extension to }\overline{U}\text{ for all }\alpha\in\N_0^d\}.$$
Equipping $C^\infty(\overline{U})$ with the family of seminorms $\{\|\cdot\|_{K,n}\colon K\subset\overline{\Omega}\text{ compact}, n\in\N_0\}$ defined by
$$\forall f\in C^\infty(\overline{U})\colon\,\|f\|_{K,n}=\max\{|\partial^\alpha f(x)|\colon x\in K, |\alpha|\leq n\}$$
turns $C^\infty(\overline{U})$ into a Fr\'echet space on which $P(D)$ acts continuously. Consequently, 
$$C^\infty_P(\overline{\Omega}):=\{g\in C^\infty(\overline{\Omega})\colon P(D)g=0\text{ in }\Omega\}$$
is a closed subspace of $C^\infty(\overline{\Omega})$ and thus a Fr\'echet space.

\begin{corollary}\label{cor-kernels}
	Let $\Omega$ be a non-empty open subset of $\R^d$. If\, $\overline{\Omega}$ satisfies the strong regularity condition then, with the notation from Proposition \ref{prop-kernels}, $\iota\colon\mathscr{E}_P(\overline\Omega)\to C^\infty_P(\overline{\Omega})$ is an isomorphism.
\end{corollary}

\begin{proof}
	Obviously, the inclusion $\mathscr{E}_{\overline{\Omega}}\hookrightarrow C^\infty(\overline{\Omega})$ is continuous. By \cite[Proposition 2.16]{Bierstone1980} it holds $C^\infty(\overline{\Omega})=\mathscr{E}_{\overline{\Omega}}$ as vector spaces and therefore, by the Open Mapping Theorem, also as Fr\'echet spaces. The claim now follows from Proposition \ref{prop-kernels}. 
\end{proof}

\section{Runge pairs for elliptic partial differential operators}\label{section:elliptic}

In this short section we will evaluate Theorem \ref{criterion-lax-malgrange} for elliptic partial differential operators in order to obtain an analogous result of the celebrated Lax-Malgrange Theorem for the setting of smooth Whitney jets. 

\begin{theorem}\label{theorem:elliptic}
	Let $P(D)$ be an elliptic partial differential operator and let $F_1\subset F_2$ be closed subsets of $\R^d$. Then $(F_1, F_2)$ is a $P$-Runge pair if and only if $F_2$ does not contain a bounded connected component of $\R^d\setminus F_1$.
\end{theorem}

\begin{proof}
	By Corollary \ref{cor:necessary condition} we only have to show that the condition on $F_1$ and $F_2$ is sufficient to be a $P$-Runge pair.
	
	Thus, we assume that $F_2$ does not contain a bounded component of $\R^d\setminus F_1$. Fix $u\in\mathscr{E}'$ with $\supp u\subset F_2$ and such that $\supp \check P(D)u\subset F_1$. We shall show that $\supp u\subset F_1$ which by Theorem \ref{criterion-lax-malgrange} will prove the theorem. Since $P(D)$ is elliptic, the same applies to $\check{P}(D)$. Since $\check P(D)u=0$ on $\R^d\setminus \supp \check P(D)u$, $u$ is a real analytic function on $\R^d\setminus \supp \check P(D)u$ (see \cite[Corollary 11.4.13]{Hor2}). Since $\supp\check{P}(D)u\subset F_1$ we get that $u$ is real analytic on $\R^d\setminus F_1$. Clearly, $u$ vanishes on $\R^d\setminus \supp u$. In particular, $u$ vanishes on the single unbounded open connected component of $\R^d\setminus F_1$.
	
	Additionally, by assumption, any bounded open connected component $B$ of $\R^d\setminus F_1$ intersects $\R^d\setminus F_2\subset \R^d\setminus\supp  u$, so $u$ vanishes also on $B$, by the real analyticity of $u$ on $\R^d\backslash F_1$. Hence $u=0$ on $\R^d\setminus F_1$, and $\supp u\subset F_1$.          
\end{proof}

\begin{example}\label{example:holomorphic}
	Let $\Omega$ be a non-empty, open subset of $\R^2\cong\C$. The Cauchy-Riemman operator $P(D)=\frac{1}{2}\left(\partial_1+i\partial_2\right)$ is elliptic and, as is well-known (cf.\ \cite[Theorem 4.4.2]{Hor1}), the space $\mathscr{E}_P(\Omega)$ and the Fr\'echet space of holomorphic functions $H(\Omega)$ on $\Omega$ (equipped, as usual, with the compact-open topology) coincide as Fr\'echet spaces.
	
	By definition, $A^\infty(\Omega)$ is the subspace of $H(\Omega)$ consisting of the holomorphic functions $g$ on $\Omega$ whose derivatives $g^{(l)}$ extend continuously to $\overline{\Omega}$, $l\in\N_0$. $A^\infty(\Omega)$ is a Fr\'echet space in a natural way, i.e.\ when equipped with the set of seminorms $\{\|\cdot\|_{m,K}\colon m\in\N_0, K\subset\overline{\Omega}\text{ compact}\}$ defined as
	$$\forall\,g\in A^\infty(\Omega)\colon\|g\|_{m,K}=\max_{0\leq l\leq m, z\in K}|g^{(l)}(z)|.$$
	In case of $\Omega=\operatorname{int}\overline{\Omega}$, it is easily seen that $A^\infty(\Omega)=C^\infty_P(\overline{\Omega})$ as Fr\'echet spaces.
	
	Hence, in case $\Omega=\operatorname{int}\overline{\Omega}$ and $\overline{\Omega}$ satisfies the strong regularity condition, by Corollary \ref{cor-kernels}, $A^\infty(\Omega)$ is canonically isomorphic to $\mathscr{E}_P(\overline{\Omega})$. By Theorem \ref{theorem:elliptic}, we thus derive for open subsets $\Omega_1, \Omega_2$ of $\C$ with $\overline{\Omega}_1\subset\overline{\Omega}_2$, $\Omega_j=\operatorname{int}\overline{\Omega}_j\, (j=1,2)$, and for which their closures satisfy the strong regularity condition, the equivalence of the following assertions.
	\begin{itemize}
		\item[(i)] The subspace $\{g_{\mid\Omega_1}\colon g\in A^\infty(\Omega_2)\}$ of $A^\infty(\Omega_1)$ is dense in $A^\infty(\Omega_1)$.
		\item[(ii)] $\overline{\Omega}_2$ does not contain any bounded connnected component of $\C\backslash\overline{\Omega}_1$.
	\end{itemize}
	In the above situation, for the special case $\Omega_2=\C$, it holds $A^\infty(\C)=H(\C)$ as Fr\'echet spaces. Since (holomorphic) polynomials are dense in $H(\C)$, we derive the following corollary which complements \cite[Theorem 5.1]{NeZa15}, generalizes \cite[Theorem 5.2]{NeZa15}, and answers the open problem (see \cite[Introduction]{GaNe16}) to characterize those domains $\Omega$ in $\C$ for which the polynomials are dense in $A^\infty(\Omega)$, under the additional hypothesis that $\overline{\Omega}$ satisfies the strong regularity condition.  
\end{example}

\begin{corollary}\label{cor:A infinity}
	Let $\Omega$ be an open subset of $\C$ with $\Omega=\operatorname{int}\overline{\Omega}$ and for which $\overline{\Omega}$ satisfies the strong regularity condition. Then the following are equivalent.
	\begin{itemize}
		\item[(i)] The subspace $\{p_{\mid\Omega}\colon p\in\C[z]\}$ of $A^\infty(\Omega)$ is dense in $A^\infty(\Omega)$.
		\item[(ii)] $\C\setminus\overline{\Omega}$ does not have a bounded connected component.
	\end{itemize} 
\end{corollary}

\begin{example}\label{example:harmonic}
	Let $\Omega$ be a non-empty, open subset of $\R^d$. The Laplace operator $\Delta=\sum_{j=1}^d\partial_j^2$ is elliptic. Therefore, as is well-known (cf.~\cite[Theorem 4.4.2]{Hor1}), the space $\mathscr{E}_\Delta(\Omega)$ and the Fr\'echet space of harmonic functions $\mathcal{H}(\Omega)$ on $\Omega$ (equipped, as usual, with the topology of local uniform convergence) coincide as Fr\'echet spaces. In analogy to Example \ref{example:holomorphic}, under the additional assumption that $\Omega=\operatorname{int}\overline{\Omega}$, we abbreviate $C_\Delta^\infty(\overline{\Omega})$ by $\mathcal{H}^\infty(\Omega)$.
	
	As in Example \ref{example:holomorphic}, in case $\Omega=\operatorname{int}\overline{\Omega}$ and $\overline{\Omega}$ satisfies the strong regularity condition, by Corollary \ref{cor-kernels}, $\mathcal{H}^\infty(\Omega)$ and $\mathscr{E}_\Delta(\overline{\Omega})$ are canonically isomorphic. Again, by Theorem \ref{theorem:elliptic}, we thus derive for open subsets $\Omega_1, \Omega_2$ of $\R^d$ with $\overline{\Omega}_1\subset\overline{\Omega}_2$, $\Omega_j=\operatorname{int}\overline{\Omega}_j\, (j=1,2)$, and for which their closures satisfy the strong regularity condition, the equivalence of the following assertions.
	\begin{itemize}
		\item[(i)] The subspace $\{g_{\mid\Omega_1}\colon g\in \mathcal{H}^\infty(\Omega_2)\}$ of $\mathcal{H}^\infty(\Omega_1)$ is dense in $\mathcal{H}^\infty(\Omega_1)$.
		\item[(ii)] $\overline{\Omega}_2$ does not contain any bounded connnected component of $\R^d\backslash\overline{\Omega}_1$.
	\end{itemize}
	Similar to Example \ref{example:holomorphic}, in the above situation, for the special case $\Omega_2=\R^d$, we have $\mathcal{H}^\infty(\R^d)=\mathcal{H}(\R^d)$ as Fr\'echet spaces. Since by \cite[Chapitre 1.2, Th\'eor\`eme 2']{Malgrange}, harmonic polynomials are dense in $\mathcal{H}(\R^d)$, the following are equivalent for an open subset $\Omega$ of $\R^d$ with $\Omega=\operatorname{int}\overline{\Omega}$ and for which $\overline{\Omega}$ satisfies the strong regularity condition.
	\begin{itemize}
		\item[(i')] The subspace $\{p_{\mid \Omega}\colon p\in\mathcal{H}(\R^d)\text{ polynomial}\}$ of $\mathcal{H}^\infty(\Omega)$ is dense in $\mathcal{H}^\infty(\Omega)$.
		\item[(ii')] $\R^d\backslash\overline{\Omega}$ does not have a bounded connected component.
	\end{itemize}
\end{example}

\section{Runge pairs for certain non-elliptic partial differential operators}\label{section:non-elliptic}

In this section we study $P$-Runge pairs for non-elliptic operators which is motivated by the results from \cite{DebKalmes2024,Kalmes2021}.

\begin{theorem}\label{thm-non-elliptic}
	Let $d\geq 2$ and let $P\in\C[\xi_1,\ldots,\xi_d]$ be a non-elliptic polynomial with principal part $P_m$. Moreover, let $F_1\subset F_2$ be closed subsets of $\R^d$. Consider the following conditions.
	\begin{enumerate}
		\item[(i)] $(F_1, F_2)$ is a $P$-Runge pair.
		\item[(ii)] There is no characteristic hyperplane $H$ for $P(D)$ such that $F_2$ contains a bounded connected component of $(\R^d\setminus F_1)\cap H$.
		\item[(iii)] There is a subspace $W\neq \{0\}$ of $\R^d$ with $\{\xi\in\R^d\colon P_m(\xi)=0\}\subset W^\perp$ satisfying the following property.
		
		For every $y\in\R^d\setminus F_1$, and $\epsilon>0$ with $B(y,\epsilon)\subset\R^d\setminus F_1$ there is $x\in\R^d$ such that $B(y,\epsilon)\cap (x+W)\neq\emptyset$ and, in case the connected component $C$ of $(\R^d\setminus F_1)\cap (x+W)$ which contains $B(y,\epsilon)\cap (x+W)$ is bounded, $C$ is not contained in $F_2$. 
		\item[(iv)] There is a subspace $W\neq \{0\}$ of $\R^d$ with $\{\xi\in\R^d\colon P_m(\xi)=0\}\subset W^\perp$ satisfying the following property.
		
		For every $x\in\R^d$, no bounded connected component of $(\R^d\setminus F_1)\cap (x+W)$ is contained in $F_2$.
	\end{enumerate}
	Then, conditions (i), (ii), and (iii) follow from (iv). In case $P$ is hypoelliptic, (iii) implies (i), too.
\end{theorem}

\begin{proof}
	First, we note that (iii) trivially follows from (iv).
	
	We continue by showing that (iv) implies (i). Thus, let $W$ be as in (iv). In view of Theorem \ref{criterion-lax-malgrange} we have to show that $\supp u\subset F_1$ whenever $u\in\mathscr{E}'(F_2)$ with $\supp \check P(D)u\subset F_1$. Thus, let us fix such $u$ as well as $x\in F_2\setminus F_1$. Let $A$ be the connected component of $(\R^d\setminus F_1)\cap (x+W)$ containing $x$. Then $A\setminus\supp u\neq\emptyset$. Indeed, while this is obvious in case $A$ is unbounded, for bounded $A$ (iv) implies that $A$ is not contained in $F_2$, so $A\setminus\supp u\supseteq A\setminus F_2\neq\emptyset$. 
	
	Next, we fix $y\in A\setminus\supp u\neq\emptyset$.
	Since $A$ is path connected (as a connected and open subset of $x+W$), there is a continuous curve $\alpha\colon[0,1]\to A$ such that $\alpha(0)=x$ and $\alpha(1)=y$. Let 
	$$\epsilon:=\min\{\operatorname{dist}(\im\alpha,F_1),\operatorname{dist}(y,\supp u)\}.$$
	Then $\epsilon>0$, $B(y,\epsilon)\subset\R^d\setminus\supp u$ and $\im\alpha+B(0,\epsilon)\subset\R^d\setminus F_1$, where $B(y,\varepsilon)$ denotes the open Euclidean ball about $y$ with radius $\varepsilon$. For the open covering $\{B(z,\epsilon)\colon z\in\im\alpha\}=\im\alpha+B(0,\epsilon)$ of the compact set $\im\alpha$ we find a finite subcovering $\{B(z_j,\epsilon)\}_{j=0}^{k}$
	for suitable $z_j\in\im\alpha$, and where -- without lost of generality -- $z_0=x$ and $z_k=y$. 
	
	Let $G$ be the graph with $\{z_0,z_1,\ldots,z_k\}$ as the set of vertices and where $z_iz_j$ is an edge if and only if the balls $B(z_i,\epsilon)$ and $B(z_j,\epsilon)$ intersect. By the connectedness of $\im\alpha$, the graph $G$ is connected and thus there is a path $\alpha(t_0),\alpha(t_1),\ldots,\alpha(t_n)$ in $G$ with $t_0=0$ and $t_n=1$. Since the line segment joining centers of two intersecting balls belongs to their union, the polygonal chain $$\mathcal{P}:=[\alpha(0),\alpha(t_1)]\cup[\alpha(t_1),\alpha(t_2)]\cup\ldots\cup[\alpha(t_{n-1}),\alpha(1)]$$ 
	is contained in 
	$$B(\alpha(0),\epsilon)\cup B(\alpha(t_1),\epsilon)\cup\ldots\cup B(\alpha(1),\epsilon)\subset\im\alpha+B(0,\epsilon).$$ 
	For 
	$$\delta:=\min\{\epsilon,\operatorname{dist}(\mathcal{P},\R^d\setminus(\im\alpha+B(0,\epsilon))\}>0$$
	we thus have 
	$$\mathcal{P}+B(0,\delta)\subset\im\alpha+B(0,\epsilon)\subset\R^d\setminus F_1.$$
	Consequently, $\check P(D)u=0$ in $\mathcal{P}+B(0,\delta)$. 
	Moreover, the convexity of $x+W$ yields $\mathcal{P}\subset x+W$.
	
	We denote the Euclidean scalar product in $\R^d$ by $\langle\cdot,\cdot\rangle$. Let $H=\{\xi\in\R^d\colon\langle\xi,N\rangle=c\}$ ($N\in\R^d, |N|=1$, $c\in\R$) be any characteristic hyperplane for $\check P(D)$ which intersects $[\alpha(t_{n-1}),\alpha(1)]+B(0,\delta)$. Then,
	$$N\in\{\xi\in\R^d\colon\check P_m(\xi)=0\}=\{\xi\in\R^d\colon P_m(\xi)=0\}\subset W^\perp,$$
	and there are $\lambda\in[0,1]$ and $z\in B(0,\delta)$ with
	\begin{align*}
		c&=\langle\lambda\alpha(t_{n-1})+(1-\lambda)\alpha(1)+z,N\rangle=\langle\alpha(1)+\lambda(\alpha(t_{n-1})-\alpha(1))+z,N\rangle\\
		&=\langle\alpha(1)+z,N\rangle+\lambda\langle\alpha(t_{n-1})-\alpha(1),N\rangle
		=\langle\alpha(1)+z,N\rangle
	\end{align*}
	so that $B(\alpha(1),\delta)\cap H\neq\emptyset$.
	
	Now, since $u$ vanishes on $B(y,\delta)=B(\alpha(1),\delta)$ and $\check P(D)u=0$ in $[\alpha(t_{n-1}),\alpha(1)]+B(0,\delta)$, it follows from \cite[Theorem 8.6.8]{Hor1} that $u$ vanishes on $[\alpha(t_{n-1}),\alpha(1)]+B(0,\delta)$ as well. Proceeding by induction we conclude that $u$ vanishes on $\mathcal{P}+B(0,\delta)$. In particular, $\supp u\subset F_2\setminus B(x,\delta)$. Since $x$ was arbitrarily chosen from the set $F_2\setminus F_1$ we obtain $\supp u\subset F_1$. This completes the proof that (iv) implies (i).
	
	Next, we prove that (iv) implies (ii). Let $W$ be again as in (iv). We assume that $F_2$ contains a bounded connected component $C$ of $(\R^d\setminus F_1)\cap H$ where $H=\{\xi\in\R^d\colon\langle\xi,N\rangle=c\}$ ($N\in\R^d, |N|=1, c\in \R$) is some characteristic hyperplane for $P(D)$. Denoting the orthogonal complement of $\operatorname{span}\{N\}$ by $H_0$, by (iv) we have $W\subset H_0$. Consequently, $W+x\subset H$ for all $x\in\R^d$ with $H=H_0+x$. In particular, for $x\in C\subset H$, $C\cap(W+x)\neq\emptyset$. Now, let $C'$ be the connected component of $C\cap(W+x)$ containing $x$. As a subset of the bounded set $C$, $C'$ is a bounded connected component of $(\R^d\setminus F_1)\cap(W+x)$ contained in $F_2$ which contradicts condition (iv).
	
	Finally, we additionally assume that $P$ is hypoelliptic. It remains to prove that (iii) implies (i). By \cite[Theorem 11.1.3]{Hor2} it follows that $\check{P}$ is hypoelliptic, too. Fix $u\in\mathscr{E}'$ with $\supp u\subset F_2$ and $\supp \check P(D)u\subset F_1$. In view of Theorem \ref{criterion-lax-malgrange}, to complete the proof we have to show $\supp u\subset F_1$.
	For this, it is enough to verify that the set 
	$$A:=\{x\in\R^d\setminus F_1\colon u \text{ vanishes on some neighborhood of } x\}$$
	is dense in $\R^d\setminus F_1$. Indeed, since $\check P(D)u=0$ on $\R^d\setminus F_1$, by hypoellipticity $u\in C^\infty(\R^d\setminus F_1)$. Thus, as a smooth function on $\R^d\setminus F_1$, if $u$ vanishes on $A$, it vanishes on the closure of $A$ in $\R^d\setminus F_1$.
	
	Let us fix $y\in\R^d\setminus F_1$. We choose $\epsilon>0$ so that $B(y,\epsilon)\subset \R^d\setminus F_1$. By assumption, there is $x\in\R^d$ with $B(y,\epsilon)\cap (x+W)\neq\emptyset$ and such that the connected component $C$ of $(\R^d\setminus F_1)\cap (x+W)$ containing $B(y,\epsilon)\cap (x+W)$ is unbounded or not contained in $F_2$. In either case, due to the compactness of $\supp u\subset F_2$, there are $z\in C, \delta>0$ such that $u$ vanishes on $B(z,\delta)\subset\R^d\setminus F_2$. By the same arguments used to prove that (iv) implies (i), we conclude that $u$ vanishes in a neighborhood of $B(y,\epsilon)\cap (x+W)$, and therefore $A\cap B(y,\epsilon)\neq\emptyset$. This means that $A$ is dense in $\R^d\setminus F_1$ and the proof is complete. 
\end{proof}

Combining Theorem \ref{thm-non-elliptic} with Corollary \ref{corollary:compositions} allows to prove the following approximation result for the one-dimensional wave operator. Compare with \cite{DebKalmes2024} for the analogous result on open subsets of $\R^2$.

\begin{corollary}\label{corollary:one dimensional wave operator}
	Let $P(D)=\partial_1^2-\partial_2^2$ be the wave operator in one spatial variable. Moreover, let $F_1\subset F_2$ be closed subsets of $\R^2$ which satisfy that for every characteristic hyperplane $H$ for $P(D)$ the set $(\R^2\backslash F_1)\cap H$ does not have a bounded connected component which is contained in $F_2$. Then, $(F_1,F_2)$ is a $P$-Runge pair.
\end{corollary}

\begin{proof}
	With $P_1(D)=\partial_1-\partial_2$ and $P_2(D)=\partial_1+\partial_2$ it holds $P(D)=P_1(D)P_2(D)$. Applying Theorem \ref{thm-non-elliptic} (iv) to $P_1$ with $W=\operatorname{span}\{(1,-1)\}$ and to $P_2$ with $W=\operatorname{span}\{(1,1)\}$, respectively, the hypothesis implies that $(F_1, F_2)$ is both a $P_1$-Runge pair and a $P_2$-Runge pair. Thus, the claim follows from Corollary \ref{corollary:compositions}.
\end{proof}

Under additional hypotheses, the sufficient condition from the previous result is also necessary, as will be shown in Corollary \ref{cor:global Runge for wave operator} below. We continue with another consequence of Theorem \ref{thm-non-elliptic}.

\begin{corollary}\label{corollary:sufficient condition for a single characteristic}
	Let $d\geq 2$ and let $P\in\C[\xi_1,\ldots,\xi_d]$ be a hypoelliptic polynomial such that for some $N\in\R^d\setminus\{0\}$ it holds
	$$\{\xi\in\R^d\colon P_m(\xi)=0\}=\operatorname{span}\{N\},$$
	where $P_m$ denotes the principal part of $P$. For $c\in\R$, let us set $H_c:=\{\xi\in\R^d\colon\langle \xi,N\rangle=c\}$.
	
	Let $F_1\subset F_2$ be closed subsets of $\R^d$ and assume that there is a dense set $D\subset\R$ such that for all $c\in D$
	no bounded connected component of $(\R^d\setminus F_1)\cap H_c$ is contained in $F_2$. 
	Then $(F_1,F_2)$ is a $P$-Runge pair.  
\end{corollary}

\begin{proof}
	Setting $W=\operatorname{span}\{N\}^\perp$, every characteristic hyperplane $H$ of $P(D)$ satisfies $H=H_c=x+W$ for suitable $c\in \R, x\in \R^d$. Thus, we see that the condition (ii) of Theorem \ref{thm-non-elliptic} is satisfied so that the claim follows from Theorem \ref{thm-non-elliptic}.
\end{proof}

Our next objective is to prove a kind of converse to the implication ''(iv)$\Rightarrow$(i)'' from Theorem \ref{thm-non-elliptic} under suitable mild additional assumptions. This requires to establish an improvement of Corollary \ref{cor:necessary condition} for certain (by Theorem \ref{theorem:elliptic} necessarily) non-elliptic polynomials. Before we do so, we introduce some notation. Clearly, for $d\geq 2$, every polynomial $P\in \C[\xi_1,\ldots,\xi_d]$ of degree $m$ can be written in a unique way as a polynomial of degree at most $m$ in the variable $\xi_d$ and with coefficients in $\C[\xi_1,\ldots,\xi_{d-1}]$, i.e.\
$$P(\xi_1,\ldots,\xi_d)=\sum_{k=0}^mQ_k(\xi_1,\ldots,\xi_{d-1})\xi_d^k,$$
where $Q_k\in\C[\xi_1,\ldots,\xi_{d-1}]$ satisfies $\operatorname{deg}Q_k\leq m-k$, $0\leq k\leq m$. A moment's reflection reveals that $Q_m$ is a non-zero, constant polynomial whenever $e_d$ is not a characteristic vector of $P$. We denote the degree of the $\xi_1$-variable in $Q_k$ by $\operatorname{deg}_{\xi_1}Q_k$ so that $\operatorname{deg}_{\xi_1}Q_k\leq \operatorname{deg}Q_k\leq m-k$ for every $0\leq k\leq m$. 

\begin{theorem}\label{thm}
	Let $d\geq 2$ and let $P$ be a polynomial of degree $m$ such that $e_1$ is characteristic for $P$ while $e_d$ is not. Moreover, assume that
	$$P(\xi_1,\ldots,\xi_d)=\sum_{k=0}^mQ_k(\xi_1,\ldots,\xi_{d-1})\xi_d^k,$$
	where $Q_k\in\C[\xi_1,\ldots,\xi_{d-1}]$ satisfies $\operatorname{deg}_{\xi_1}Q_k< m-k$ for every $0\leq k\leq m-1$.
	
	Additionally, let the pair $(F_1,F_2)$ of closed subsets $F_1\subset F_2$ of $\R^d$ be a $P$-Runge pair. Then there are no $c\in\R$, $\epsilon>0$ for which $F_2$ contains a bounded connected component of $(\R^d\setminus F_1)\cap (\left(c-\epsilon,c+\epsilon\right)\times\R^{d-1})$. 
\end{theorem}

We would like to point out that the hypotheses on $P$ in Theorem \ref{thm} are satisfied by polynomials of the form $P(\xi_1,\ldots,\xi_d)=c\xi_1^k+R(\xi_2,\ldots,\xi_d)$, where $R$ is an elliptic polynomial in $d-1$ variables of degree strictly larger than $k\in\N_0$ and $c\in\C$. In particular, this covers parabolic partial differential operators/polynomials with $\xi_1$ in the r\^{o}le of the time variable like the heat operator $P(D)=\partial_1-\sum_{j=2}^d\partial_j^2$, or the time-dependent free Schr\"odinger operator, again with $\xi_1$ in the r\^{o}le of the time variable, $P(D)=i\partial_1+\sum_{j=2}^d \partial_j^2$. 

Additionally, under the hypotheses of Theorem \ref{thm}, for every $c\in\R$ the hyperplane $H_c=\{x\in\R^d\colon x_1=c\}$ is characteristic for $P$. Obviously, for every $\varepsilon>0$ it holds
$$\left(c-\epsilon,c+\epsilon\right)\times\R^{d-1}=H_c+B(0,\epsilon),$$
so that Theorem \ref{thm} is indeed a weak converse to Corollary \ref{corollary:sufficient condition for a single characteristic}.

\begin{proof}[Proof of Theorem \ref{thm}]
	We argue by contradiction. Let us assume that there are $c\in\R, \epsilon>0$ and a bounded connected component $C$ of $(\R^d\setminus F_1)\cap (\left(c-\epsilon,c+\epsilon\right)\times\R^{d-1})$ contained in $F_2$. By \cite[Theorem 5]{Kalmes2021} (see also the proof of \cite[Theorem 2]{Kalmes2021}), there is a smooth function $v$ on $\R^d$ which is real analytic on $(c-\frac{\epsilon}{2},c+\frac{\epsilon}{2})\times\R^{d-1}$ and satisfies $\check P(D)v=0$ and
	$$\left[c-\frac{\epsilon}{2},c+\frac{\epsilon}{2}\right]\times\R^{d-1}\subset\supp v\subset\left[c-\epsilon,c+\epsilon\right]\times\R^{d-1}.$$
	Let us define $w\in\mathscr{D}'(\R^d)$ by
	$$\forall\,\phi\in\mathscr{D}(\R^d)\colon w(\phi):=\int_C v(x)\phi(x)dx.$$
	Clearly, $\supp w\subset C\subset F_2$, so that actually $w\in\mathscr{E}'$, and 
	$$\emptyset\neq C\cap\left(\left[c-\frac{\epsilon}{2},c+\frac{\epsilon}{2}\right]\times\R^{d-1}\right)\subset\supp w$$
	so $\supp w\nsubseteq F_1$. We shall show that $\supp \check P(D)w\subset F_1$ which by Theorem \ref{criterion-lax-malgrange} will give the desired contradiction. 
	
	For $\phi\in \mathscr{D}(C)$, integration by parts yields
	\begin{align*}
		\check P(D)w(\phi)&=\int_C v(x)P(D)\phi(x)dx=\int_{(c-\varepsilon,c+\varepsilon)\times\R^{d-1}}v(x)P(D)\phi(x)dx\\
		&=\int_{(c-\varepsilon,c+\varepsilon)\times\R^{d-1}}\check{P}(D)v(x)\phi(x)dx=0.
	\end{align*}
	We shall now examine the boundary points of $C$ which are not in $F_1$. Let us fix $a\in \partial{C}\setminus F_1\subset \overline C\cap(\{c-\epsilon,c+\epsilon\}\times\R^{d-1})$ and $\delta>0$ such that $B(a,\delta)\subset\R^d\setminus F_1$. Then, since $v$ vanishes outside the slab $[c-\epsilon,c+\epsilon]\times\R^{d-1}$, integration by parts gives for all $\phi\in \mathscr{D}(B(a,\delta))$
	\begin{align*}
		\check P(D)w(\phi)&=\int_C v(x)P(D)\phi(x) dx=\int_{C\cap B(a,\delta)} v(x)P(D)\phi(x)dx\\
		&=\int_{C\cap B(a,\delta)}\check P(D)v(x)\phi(x)dx=0. 
	\end{align*}
	Since $a\in\partial C\setminus F_1$ was arbitrary, this shows that $\supp \check P(D)w$ vanishes on 
	$$(\R^d\setminus \overline{C})\cup C\cup (\partial C\setminus F_1)=\R^d\setminus(\partial C\cap F_1)$$
	implying $\supp \check P(D)w\subset\partial C\cap F_1\subset F_1$. 
\end{proof}

We shall now improve the necessary condition for $P$-Runge pairs from Theorem \ref{thm} for the special case of $F_2=\R^d$ and under suitable additional hypotheses on the boundary of $F_1$. For the special case of $d=2$ we have the following proposition.

\begin{prop}\label{prop:continuous boundary for d=2}
	Let $F\subset \R^2$ be closed. Moreover, assume that the following conditions are satisfied.
	\begin{itemize}
		\item[(a)] $F$ has continuous boundary, i.e.\ for each $\xi\in\partial F$ there is an open neighborhood $U_{\xi}$ in $\R^2$ and a homeomorphism $h_{\xi}\colon B(0,1)\to U_{\xi}$ such that $h_{\xi}(0)=\xi$ and
		$$\partial F\cap U_{\xi}=h_{\xi}\left(\{x\in B(0,1)\colon x_2=0\}\right).$$ 
		\item[(b)] For each $\xi\in\partial F$ and every neighborhood $U$ of $\xi$ in $\R^2$ it holds
		$$U\cap\partial F\cap\{x\in\R^2\colon x_1<\xi_1\}\neq\emptyset \text{ and }U\cap\partial F\cap\{x\in\R^2\colon x_1>\xi_1\}\neq\emptyset.$$	
	\end{itemize}
	Then, for every $c\in \R$ for which $(\R^2\backslash F)\cap \{c\}\times\R$ has a bounded connected component $C$, there is $\delta>0$ such that the connected component of $\left(\R^2\backslash F\right)\cap (c-\delta,c+\delta)\times\R$ which contains $C$ is bounded, too. 
\end{prop}

\begin{proof}
	Let $c\in \R$ be such that $(\R^2\backslash F)\cap\{c\}\times\R$ ha a bounded connectd component $C$. The boundary $\partial_c C$ of $C$ in $\{c\}\times\R$ consists of two points $\xi_1,\xi_2\in\partial F$. Let $U_j:=U_{\xi_j}$ and $h_j:=h_{\xi_j}$ as in hypothesis (a), $j=1,2$.
	
	Obviously,
	$$\gamma_j\colon (-1,1)\rightarrow \partial F\cap U_j, y\mapsto h_j(y,0)$$
	is a correctly defined homeomorphism. With $\pi_1\colon\R^2\rightarrow \R, x\mapsto x_1$ we conclude that
	$$I_j:=\pi_1\circ \gamma_j\left((-1,1)\right)$$
	is an interval. Obviously, $c=\pi_1(\xi_j)=\pi_1(\gamma_j(0))\in I_j$. We claim that $c$ is an interior point of $I_j$. Indeed, assuming that this is not the case, we either have $I_j\subset(-\infty,c]$ or $I_j\subset[c,\infty)$. In the first case,
	$$(-\infty,c]\supset \pi_1(\partial F\cap U_j)$$
	so that $\partial F\cap U_j\cap \{x\in\R^2\colon x_1>c\}$ is empty, contradicting hypothesis (b), as $c=\pi_1(\xi_j)$. Similarly, the assumption $I_j\subset[c,\infty)$ leads to a contradiction. Thus, there is $\delta'>0$ such that
	$$[c-\delta', c+\delta']\subset I_j=\pi_1\left(\partial F\cap U_j\right), j=1,2.$$
	
	Next, we fix $\varepsilon\in (0,\delta')$ satisfying $$[\pi_1(\xi_j)-\varepsilon,\pi_1(\xi_j)+\varepsilon]\times[\pi_2(\xi_j)-\varepsilon,\pi_2(\xi_j)+\varepsilon]\subset U_j, j=1,2.$$
	Then,
	$$C\cap \left(\R^2\backslash\cup_{j=1,2}(\pi_1(\xi_j)-\varepsilon,\pi_1(\xi_j)+\varepsilon)\times(\pi_2(\xi_j)-\varepsilon,\pi_2(\xi_j)+\varepsilon)\right)$$
	is a closed subset of $\R^2$ which is disjoint to $F$, and, by making $\varepsilon$ smaller if necessary, is non-empty. Hence, its Euclidean distance $\delta''$ to $F$ is finite and strictly positive. 
	
	We set $\delta:=\min\{\varepsilon,\delta''/2\}$. For $j=1,2$, there are unique $$\xi_j^+,\xi_j^-\in \partial F\cap [\pi_1(\xi_j)-\varepsilon,\pi_1(\xi_j)+\varepsilon]\times[\pi_2(\xi_j)-\varepsilon,\pi_2(\xi_j)+\varepsilon]$$
	satisfying $\pi_1(\xi_j^+)=c+\delta$ and $\pi_1(\xi_j^-)=c-\delta$. Additionally, for suitable $t_j^+\in(0,1)$ and $t_j^-\in (-1,0)$ it holds $\gamma_j(t_j^+)=\xi_j^+$ and $\gamma_j(t_j^-)=\xi_j^-$, and the continuous curves
	$$\gamma^+:[0,1]\rightarrow\R^2, t\mapsto \xi_1^+ + t(\xi_2^+-\xi_1^+)$$
	and
	$$\gamma^-:[0,1]\rightarrow\R^2, t\mapsto \xi_2^- + t(\xi_1^- -\xi_2^-)$$
	satisfy $\gamma^+((0,1))\cap\partial F=\emptyset=\gamma^-((0,1))\cap\partial F$, $\pi_1(\gamma^+(t))=c+\delta$, and $\pi_1(\gamma^-(t))=c-\delta$, $t\in [0,1]$. The concatenated curve
	$$\gamma:[0,4]\rightarrow\R^2,t\mapsto\begin{cases}
		\gamma^+(t),&\text{if }t\in [0,1],\\
		\gamma_2(t_2^+ +(t-1)(t_2^--t_2^+)),&\text{if }t\in[1,2],\\
		\gamma^-(t-2),&\text{if }t\in[2,3],\\
		\gamma_1(t_1^-+(t-3)(t_1^+-t_1^-)),&\text{if }t\in[3,4]
	\end{cases}$$
	is correctly defined, continuous, closed, and $\gamma_{|[0,4)}$ is injective. The bounded connected component $B$ of $\R^2\backslash\gamma([0,4])$ contains $C$ and coincides with the connected component of $(\R^2\backslash F)\cap (c-\delta,c+\delta)\times \R$ containing $C$. This proves the assertion.
\end{proof}

The following simple example shows that the conclusion of Proposition \ref{prop:continuous boundary for d=2} is not true if we drop hypothesis (b).

\begin{example}
	Let $F$ be the union of the two circles $\partial B(0,1)=\{x\in \R^2: x_1^2+x_2^2=1\}$ and $\partial B((0,2),1)=\{x\in \R^2: x_1^2+(x_2-2)^2=1\}$. Then,
	$$(\R^2\setminus F)\cap\{1\}\times\R=\{1\}\times (-\infty,0)\cup \{1\}\times (0,2)\cup \{1\}\times (2,\infty)$$
	has a bounded connected component but there is no $\delta>0$ such that the same is true for
	$$(\R^2\setminus F)\cap(1-\delta,1+\delta)\times\R.$$
\end{example}

In order to prove a similar result to the previous one for $d\geq 3$, for technical reasons, we have to assume that $\partial F$ is a $C^1$-hypersurface in $\R^d$.

\begin{prop}\label{prop:C1-boundary, d>2}
	Assume that $d\geq 3$ and that $F\subset\R^d$ is closed such that $\partial F$ is a $C^1$-hypersurface with the property that for every $\xi\in \partial F$ its normal space is not spanned by $e_1$. 
	
	Moreover, let $c\in\R$ be such that $\left(\R^d\backslash F\right)\cap \{c\}\times\R^{d-1}$ has a bounded connected component $C$. Then, there is $\delta>0$ such that the connected component of $\left(\R^d\backslash F\right)\cap (c-\delta,c+\delta)\times\R^{d-1}$ which contains $C$, is bounded, too. 
\end{prop}

\begin{proof}
Let $c\in\R$ and $C\subset \left(\R^d\backslash F\right)\cap \{c\}\times\R^{d-1}$ be as in the hypothesis. We denote the boundary of $C$ in $\{c\}\times\R^{d-1}$ by $\partial_c C$. Then, $\partial_c C$ is a compact subset of $\{c\}\times\R^{d-1}$ and $\left(\{c\}\times\R^{d-1}\right)\setminus \partial_c C$ has exactly one unbounded connected component which we denote by $D_\infty$. In particular, $D_\infty$ is open in $\{c\}\times\R^{d-1}$, disjoint to $C$, and 
\begin{equation}\label{eq:set inclusion 1}
	\partial_cD_\infty\subset\partial_cC\subset\left(\{c\}\times\R^{d-1}\right)\cap \partial F,
\end{equation}
where $\partial_c D_\infty$ is the boundary of $D_\infty$ in $\{c\}\times\R^{d-1}$. 

Let us fix $\xi\in\partial_c D_\infty$. By hypothesis, there is an open, connected subset $U_\xi$ of $\R^d$ conataining $\xi$ and a $C^1$-function $f_\xi:U_\xi\rightarrow \R$ such that
	$$\partial F\cap U_\xi=\{x\in U_\xi \colon f_\xi(x)=0\}$$
	and $e_1, \nabla f_\xi(x)$ are linearly independent for each $x\in U_\xi$. 
Then
$\frac{\partial f_\xi}{\partial x_j}(\xi)\neq0$ for some $2\leq j\leq d$, and w.l.o.g. we may assume that $\frac{\partial f_\xi}{\partial x_d}(\xi)\neq0$. 
From the implicit function theorem it follows that there is $\delta_\xi>0$ for which the open sets
$$Q_\xi:=(\xi_1-\delta_\xi,\xi_1+\delta_\xi)\times\ldots\times(\xi_{d-1}-\delta_\xi,\xi_{d-1}+\delta_\xi)$$
and
$$P_\xi:=(\xi_d-\delta_\xi,\xi_d+\delta_\xi)$$
satisfy $V_\xi:=Q_\xi\times P_\xi\subset U_\xi$ and there is a unique
$C^1$-function $g_\xi\colon Q_\xi\to P_\xi$,
such that
$$g_\xi(\xi_1,\ldots,\xi_{d-1})=\xi_d\quad\text{and}\quad f_\xi(x,g_\xi(x))=0$$
for all $x\in Q_\xi$.
Then 
$$\partial F\cap V_\xi=\{(x,y)\in Q_\xi\times P_\xi\colon f_\xi(x,y)=0\}=\{(x,g_\xi(x))\colon x\in Q_\xi\},$$ 
so
$$V_\xi\setminus\partial F=V_\xi^-\cup V_\xi^+,$$
where the sets
$$V_\xi^-:=\{(x,y)\in Q_\xi\times P_\xi\colon y<g_\xi(x)\}\quad\text{and}\quad
V_\xi^+:=\{(x,y)\in Q_\xi\times P_\xi\colon y>g_\xi(x)\}$$
are non-empty, pairwise disjoint, open and connected. Therefore, the set $V_\xi\setminus\partial F$ has exactly two open connected components $V_\xi^-$ and $V_\xi^+$.

Similarly, the set $(V_\xi\cap\{c\}\times\R^{d-1})\setminus\partial F$ is divided by the graph of a $C^1$-function $h_\xi\colon \widehat{Q_\xi}\to P_\xi$, $h_\xi(x_2,\ldots,x_{d-1})=g_\xi(\xi_1,x_2,\ldots,x_{d-1})$ into two non-empty, open, connected sets, where
$$\widehat{Q_\xi}=(\xi_2-\delta_\xi,\xi_2+\delta_\xi)\times\ldots\times(\xi_{d-1}-\delta_\xi,\xi_{d-1}+\delta_\xi),$$
and the decomposition is the following
$$(V_\xi\cap\{c\}\times\R^{d-1})\setminus\partial F=
(V_\xi^-\cap\{c\}\times\R^{d-1})\cup (V_\xi^+\cap\{c\}\times\R^{d-1}).$$

Let us note that the sets $V_\xi\cap D_\infty$ and $V_\xi\cap C$ are pairwise disjoint non-empty open subsets of $V_\xi\cap\{c\}\times\R^{d-1}$. 
We claim that $(V_\xi\cap\{c\}\times\R^{d-1})\setminus\partial F\subset C\cup D_\infty$. Indeed, if this inclusion were not true, then we would find a bounded connected component $B$ of $\left(\{c\}\times\R^{d-1}\right)\setminus \partial_c C$ which is different from $C$ and such that 
$V_\xi\cap B$ is a non-empty open set. Then $(V_\xi\cap\{c\}\times\R^{d-1})\setminus\partial F$ would have at least three connected components as none two sets among $V_\xi\cap D_\infty$, $V_\xi\cap C$ and $V_\xi\cap B$ are in the same connected component, a contradiction. The inclusion just demonstrated leads easily to the identity  
$$(V_\xi\cap\{c\}\times\R^{d-1})\setminus\partial F=(V_\xi\cap C)\cup (V_\xi\cap D_\infty).$$
Since, $(V_\xi\cap\{c\}\times\R^{d-1})\setminus\partial F$ has exactly two open connected components, we have
$$V_\xi\cap C=V_\xi^-\cap\{c\}\times\R^{d-1}\quad\text{and}\quad V_\xi\cap D_\infty=V_\xi^+\cap\{c\}\times\R^{d-1}$$ 
or 
$$V_\xi\cap C=V_\xi^+\cap\{c\}\times\R^{d-1}\quad\text{and}\quad V_\xi\cap D_\infty=V_\xi^-\cap\{c\}\times\R^{d-1}.$$ 

Since $\partial_c D_\infty\subset \partial_c C$ we have that $\partial_c D_\infty$ is compact and $\{V_\xi:\xi\in\partial_c D_\infty\}$ is an open covering of $\partial_c D_\infty$. Additionally, for every $\xi\in \partial_c D_\infty$ it holds
	$$V_\xi\cap\partial F\cap\left(\{c\}\times\R^{d-1}\right)=V_\xi\cap \partial_c D_\infty.$$
Indeed, ''$\supset$'' follows directly from \eqref{eq:set inclusion 1}. In order to show ''$\subset$'', we observe
\begin{eqnarray*}
	\overline{V_\xi^+\cap(\{c\}\times\R^{d-1})}^{\{c\}\times\R^{d-1}}&=&\overline{\{(c,x_2,\ldots,x_d): x_d> g_\xi(c,x_2,\ldots,x_{d-1})\}}^{\{c\}\times\R^{d-1}}\\
	&\supset& \{(c,x_2,\ldots,x_d): x_d= g_\xi(c,x_2,\ldots,x_{d-1})\}\\
	&=&V_\xi\cap \partial F\cap(\{c\}\times\R^{d-1})
\end{eqnarray*}
where closures are taken in $\{c\}\times\R^{d-1}$. Thus, in case of $V_\xi\cap D_\infty=V_\xi^+\cap(\{c\}\times\R^{d-1})$ we conclude
$$V_\xi\cap \partial F\cap(\{c\}\times\R^{d-1})\subset \overline{V_\xi\cap D_\infty}^{\{c\}\times\R^{d-1}}\subset \overline{D_\infty}^{\{c\}\times\R^{d-1}}=\partial_c D_\infty\cup D_\infty$$
which by $\partial F\cap D_\infty=\emptyset$ implies
$$V_\xi\cap\partial F\cap(\{c\}\times\R^{d-1})\subset V_\xi\cap \partial_c D_\infty.$$
Analogously, in case of $V_\xi\cap D_\infty=V_\xi^-\cap(\{c\}\times\R^{d-1})$, one uses with the same arguments
$$\overline{V_\xi^-\cap(\{c\}\times\R^{d-1})}^{\{c\}\times\R^{d-1}}\supset V_\xi\cap \partial F\cap(\{c\}\times\R^{d-1})$$
to conclude again $V_\xi\cap\partial F\cap(\{c\}\times\R^{d-1})\subset V_\xi\cap \partial_c D_\infty$, as desired.

Let $\xi^{(1)},\ldots,\xi^{(m)}\in\partial_c D_\infty$ be such that $\{V_{\xi^{(1)}},\ldots, V_{\xi^{(m)}}\}$ is an open covering of $\partial_c D_\infty$. We define
	$$\delta:=\min\{\delta_{\xi^{(j)}}: 1\leq j\leq  m\}\text{ and }  V_j:=V_{\xi^{(j)}}\cap((c-\delta,c+\delta)\times\R^{d-1})$$
for $j=1,\ldots,m$.
Then, $\{V_1,\ldots, V_m\}$  
remains an open covering of $\partial _c D_\infty$ and $B:=\cup_{j=1}^m V_j\cap \partial F$ is a subset of $(c-\delta,c+\delta)\times \R^{d-1}$ which satisfies
$$B\cap \{c\}\times\R^{d-1}=\cup_{j=1}^m V_j\cap \partial F\cap(\{c\}\times\R^{d-1})=\cup_{j=1}^m V_j\cap \partial_c D_\infty=\partial_c D_\infty.$$
Let $C^\delta$ be the connected component of $(\R^d\backslash B)\cap (c-\delta,c+\delta)\times\R^{d-1}$ which contains $C$. Additionally, since $D_\infty$ is connected and obviously contained in any unbounded connected component of $(\R^d\backslash B)\cap (c-\delta,c+\delta)\times\R^{d-1}$, 
there is precisely one such component. We denote it by $D_\infty^\delta$. Then, denoting the boundary of $D_\infty^\delta$ in $(c-\delta,c+\delta)\times \R^{d-1}$ by $\partial_\delta D_\infty^\delta$ it holds
	$$B\cap\{c\}\times\R^{d-1}=\partial_c D_\infty\subset \partial_\delta D_\infty^\delta\subset\partial F.$$ 

Without loss of generality we may assume that $\delta_{\xi_1}=\delta$ which gives $V_1=V_{\xi_1}$, and also that $V_1\cap C=V_1^-\cap\{c\}\times\R^{d-1}$ and $V_1\cap D_\infty=V_1^+\cap\{c\}\times\R^{d-1}$. Then $V_1^-$ is a connected subset of 
$$(\R^d\backslash B)\cap (c-\delta,c+\delta)\times\R^{d-1}$$ which intersects $C$. Therefore, $V_1^-\subset C^\delta$ so that $V_1^-\subset V_1\cap C^\delta$. Similarly, $V_1^+\subset V_1\cap D_\infty^\delta$. From $V_1\setminus\partial F=V_1^+\cup V_1^-$ and
$$D_\infty^\delta\cup C^\delta\subset (\R^d\backslash B)\cap (c-\delta,c+\delta)\times\R^{d-1}$$
we derive $V_1\cap C^\delta\subset V_1^+\cup V_1^-$ and $V_1\cap D_\infty^\delta\subset V_1^+\cup V_1^-$. Combining these inclusions with the disjointness of $C^\delta$ and $D_\infty^\delta$ as well as  $V_1^-\subset V_1\cap C^\delta$ and $V_1^+\subset V_1\cap D_\infty^\delta$, respectively, yields $V_1\cap C^\delta=V_1^-$ and $V_1\cap D_\infty^\delta=V_1^+$, respectively.

In order to complete the proof, we have to show that $C^\delta$ and $D_\infty^\delta$ are disjoint. But this is not the case, as $C^\delta=D_\infty^\delta$ gives
$$V_1^-=V_1\cap C^\delta=V_1\cap D_\infty^\delta=V_1^+,$$
a contradiction. 
\end{proof}

With the two previous propositions we can now prove the next corollary.

\begin{corollary}\label{corollary:improvement of thm}
	Let $d\geq 2$, let $P$ be as in Theorem \ref{thm}, and let $F\subset\R^d$ be closed such that $(F, \R^d)$ is a $P$-Runge pair. Moreover, assume that one of the following additional hypotheses holds true.
	\begin{itemize}
		\item[(a)] Let $d\geq 3$ and let $F$ have $C^1$-boundary such that for every $\xi\in\partial F$ its normal space is not spanned by $e_1$.
		\item[(b)] Let $d=2$, and let $F$ have continuous boundary such that for every $\xi\in\partial F$ and every neighborhood $U$ of $\xi$ in $\R^2$ it holds
		$$U\cap \partial F\cap\{x\in \R^2: x_1 <\xi_1\}\neq \emptyset \quad\textit{and}\quad U\cap \partial F\cap\{x\in \R^2: x_1 >\xi_1\}\neq \emptyset.$$
	\end{itemize}
	Then, there is no $c\in\R$ such that $(\R^d\backslash F)\cap \{c\}\times\R^{d-1}$ has a bounded connected component.
\end{corollary}

\begin{proof}
	We assume that there is $c\in\R$ such that
	$$(\R^d\backslash F)\cap\{c\}\times\R^{d-1}$$
	has a bounded connected component. Then, by Propositions \ref{prop:C1-boundary, d>2} and \ref{prop:continuous boundary for d=2}, respectively, there is $\delta>0$ for which
	$$(\R^d\backslash F)\cap(c-\delta,c+\delta)\times\R^{d-1}$$
	has a bounded connected component, too. Since $(F, \R^d)$ is supposed to be a $P$-Runge pair, Theorem \ref{thm} gives a contradiction.
\end{proof}

We are now ready to give the proof of Theorem B from the introduction.

\begin{proof}[Proof of Theorem B]
	By Theorem \ref{thm-non-elliptic}, for $W=\operatorname{span}\{e_1\}$ in condition (iv), (i) implies (ii). Moreover, (iii) follows from (ii) by Theorem \ref{thm}.
	
	We assume that one of the additional hypotheses (a) or (b) holds. Then, invoking Corollary \ref{corollary:improvement of thm} completes the proof.
\end{proof}

We conclude with an application of Corollary \ref{corollary:improvement of thm} to the wave operator in one spatial variable. As is customary, characteristic hyperplanes in $\R^2$ are called characteristic lines.

\begin{corollary}\label{cor:global Runge for wave operator}
	Let $P(D)=\partial_1^2-\partial_2^2$ and let $F\subset\R^2$ be closed with continuous boundary such that for each $\xi\in\partial F$ and every neighborhood $U$ of $\xi$ in $\R^2$ it holds
	$$U\cap\partial F\cap\{x\in \R^2:x_1+x_2<\xi_1+\xi_2\}\neq \emptyset\text{ and }U\cap\partial F\cap\{x\in \R^2:x_1-x_2<\xi_1-\xi_2\}\neq \emptyset$$
	as well as
	$$U\cap\partial F\cap\{x\in \R^2:x_1+x_2>\xi_1+\xi_2\}\neq \emptyset\text{ and }U\cap\partial F\cap\{x\in \R^2:x_1-x_2>\xi_1-\xi_2\}\neq \emptyset.$$
	Then, the following are equivalent.
	\begin{itemize}
		\item[(i)] $(F,\R^2)$ is a $P$-Runge pair.
		\item[(ii)] There is no characteristic line for $P(D)$ for which its intersection with $\R^2\backslash F$ has a bounded connected component.
	\end{itemize}
\end{corollary}

\begin{proof}
	That (ii) implies (i) follows immediately from Corollary \ref{corollary:one dimensional wave operator}.
	
	To prove the reverse implication, we consider the linear bijection
	$$T:\R^2\rightarrow\R^2, (y_1,y_2)\mapsto\frac{1}{2} (y_1+y_2, y_1-y_2)$$
	and its inverse $T^{-1}=T$, as well as $Q(D)=\frac{\partial^2}{\partial y_1\partial y_2}$. With these, it holds $(P(D)f)\circ T=Q(D)(f\circ T)$ for every Whitney jet $f=(f^{(\alpha)})_{\alpha\in\N_0^2}$, where $f\circ T=(f^{(\alpha)}\circ T)_{\alpha\in\N_0^2}$. The isomorphims $\mathscr{E}(\R^2)\rightarrow\mathscr{E}(\R^2), f\mapsto f\circ T$ and $\mathscr{E}(T^{-1}(F))\rightarrow\mathscr{E}(F), f\mapsto f\circ T$ map $\mathscr{E}_Q(\R^2)$ and $\mathscr{E}_Q(T^{-1}(F))$ onto $\mathscr{E}_P(\R^2)$ and $\mathscr{E}_P(F)$, respectively, and they commute with the restriction map. Therefore, $(F,\R^2)$ is a $P$-Runge pair precisely when $T^{-1}(F), \R^2$ is a $Q$-Runge pair. By Corollary \ref{corollary:compositions}, the latter holds precisely when $(T^{-1}(F), \R^2)$ is both a $\frac{\partial}{\partial y_1}$-Runge pair as well as a $\frac{\partial}{\partial y_2}$-Runge pair.
	
	It is straightforward to verify that the hypotheses on $F$ imply that $T^{-1}(F)$ is closed set with continuous boundary such that for each $\eta\in\partial T^{-1}(F)$ and every neighborhood $V$ of $\eta\in\R^2$ it holds 
	$$V\cap\partial T^{-1}(F)\cap\{x\in\R^2: x_1<\eta_1\}\neq \emptyset\text{ and }V\cap\partial T^{-1}(F)\cap\{x\in\R^2: x_1>\eta_1\}\neq \emptyset$$
	as well as
	$$V\cap\partial T^{-1}(F)\cap\{x\in\R^2: x_2<\eta_2\}\neq\emptyset\text{ and }V\cap\partial T^{-1}(F)\cap\{x\in\R^2: x_2>\eta_2\}\neq\emptyset$$
	Thus, if $(T^{-1}(F), \R^2)$ is a $\frac{\partial}{\partial y_2}$-Runge pair it follows from Corollary \ref{corollary:improvement of thm} that there is no $c\in\R$ for which $(\R^2\backslash T^{-1}(F))\cap \{c\}\times\R$ has a bounded connected component. A second application of Corollary \ref{corollary:improvement of thm} combined with another obvious change of variables yields that there is no $c\in\R$ for which $(\R^2\backslash T^{-1}(F))\cap \R\times\{c\}$ has a bounded connected component whenever $(T^{-1}(F), \R^2)$ is a $\frac{\partial}{\partial y_1}$-Runge pair. Hence, if $(F,\R^2)$ is a $P$-Runge pair, there is no $c\in\R$ such that
	$$(\R^2\backslash T^{-1}(F))\cap \{c\}\times\R=T^{-1}\left((\R^2\backslash F)\cap T(\{c\}\times\R)\right)$$
	or
	$$(\R^2\backslash T^{-1}(F))\cap \R\times \{c\}=T^{-1}\left((\R^2\backslash F)\cap T(\R\times \{c\})\right)$$
	has a bounded connected component. Since the characteristic lines for $P$ are precisely given by $T(\{c\}\times\R)$ and $T(\R\times \{c\})$, $c\in\R$, it herefrom follows that (i) implies (ii).
\end{proof}

\section{Appendix}

We say that an open set $\Omega\subset\R^d$ has a \emph{H\"older continuous boundary} if for all $x_0\in\partial\Omega$ there are an open neighborhood $U$ of $x_0$, a rotation $R$ of $\R^d$, and a function $\phi\colon\R^{d-1}\to\R$ such that
\begin{equation}\label{eq:Hoelder continuity}
	\exists\, C\geq 1, \theta\in (0,1]\,\forall\, \xi,\eta\in\R^{d-1}\colon\,|\varphi(\xi)-\varphi(\eta)|\leq C|\xi-\eta|^\theta
\end{equation}
and
$$\Omega\cap U=R(\{(\xi,\eta)\colon \eta>\phi(\xi)\})\cap U.$$

\begin{prop}\label{prop:Hoelder boundary implies strong regularity}
Let $\Omega$ be an open subset of $\R^d$ such that $\overline{\Omega}$ is path connected. If $\Omega$ has H\"older continuous boundary then $\overline{\Omega}$ satisfies the strong regularity condition.
\end{prop}
\begin{proof}
Let us assume that $\Omega$ has a H\"older continuous boundray and
let us fix $x_0\in\overline{\Omega}$. If $x_0\in\mathrm{int}\Omega$, then clearly there is $0<\epsilon\leq\frac{1}{2}$ such that $B(x_0,\epsilon)\subset\mathrm{int}\Omega$, and for each pair of $x,y\in B(x_0,\epsilon)$ the segment $[x,y]$ -- whose lenght is $|x-y|\leq|x-y|^\theta$ for every $\theta\in (0,1]$ -- is in $B(x_0,\epsilon)$.  

If $x_0\in\partial\Omega$ then let us take, according to our assumption, an open neighborhood $U$ of $x_0$, a rotation $R$ of $\R^d$, and a function $\phi\colon\R^{d-1}\to\R$ such that \eqref{eq:Hoelder continuity} holds and
$$\Omega\cap U=R(\{(\xi,\eta)\colon \eta>\phi(\xi)\})\cap U.$$
Then $x_0=R(x_0',\phi(x_0'))$ for some $x_0'\in\R^{d-1}$ 
and one can choose $0<\epsilon_0\leq\frac{1}{2}$ such that
$$V:=R\left(B(x_0',\epsilon_0)\times\left(\phi(x_0')-5C\epsilon_0^\theta,\phi(x_0')+5C\epsilon_0^\theta\right)\right)\subset U.$$
Since
$$\phi(x_0')-\phi(\xi)\leq |\phi(x_0')-\phi(\xi)|\leq C|x_0'-\xi|^\theta<C\epsilon_0^\theta$$
for $\xi\in B(x_0',\epsilon_0)$,
we have clearly
$\phi(\xi)>\phi(x_0')-5C\epsilon_0^\theta$ for $\xi\in B(x_0',\epsilon_0)$.
Consequently, 
$$W:=R\left(\left\{(\xi,\eta)\colon \xi\in B(x_0',\epsilon_0)\text{ and }\phi(\xi)<\eta<\phi(x_0')+5C\epsilon_0^\theta\right\}\right)\subset V,$$
and $W\subset\Omega$.

Now let us choose $0<\epsilon<\epsilon_0$ so that $B(x_0,\epsilon)\subset V$ and fix $x,y\in\overline{\Omega}\cap B(x_0,\epsilon)$. 
Let us define, 
$$u(x',y'):=R((x',\phi(x')+2C|x'-y'|^\theta))$$ 
and
$$v(x',y'):=R((y',\phi(x')+2C|x'-y'|^\theta)).$$

If $x,y\in\partial\Omega$ then $x=R(x',\phi(x'))$ and $y=R(y',\phi(y'))$ for some $x',y'\in B(x_0',\epsilon_0)$.
If $\gamma$ is the polygonal chain $(x,u(x',y'),v(x',y'),y)$
then $\im\gamma\setminus\{x,y\}\subset W\subset \Omega$. Indeed, let us note that
\begin{align*}
\phi(x')-\phi(x_0')+2C|x'-y'|^\theta&\leq C|x'-x_0'|^\theta+2C|x'-y'|^\theta\\ 
&\leq C|x-x_0|^\theta+2C|x-y|^\theta
<C\epsilon^\theta+2C(2\epsilon)^\theta\\
&\leq 5C\epsilon^\theta<5C\epsilon_0^\theta,
\end{align*}
so $\phi(x')+2C|x'-y'|^\theta<\phi(x_0')+5C\epsilon_0^\theta$. This shows that $u(x',y')\in W$, and the segment $(x,u(x',y')]$ is contained in $\Omega$. Since
$$(1-t)\phi(x')+t\phi(y')-\phi(x')=t(\phi(y')-\phi(x'))\leq tC|x'-y'|^\theta<2C|x'-y'|^\theta$$
for all $0\leq t\leq1$, we have
$$\phi(x')+2C|x'-y'|^\theta>(1-t)\phi(x')+t\phi(y')$$
for the same parameters $t$.
Consequently, the segments $[u(x',y'),v(x',y')]$ and $[v(x',y'),y)$ are contained in $\Omega$.

Moreover,
\begin{align*}
|\gamma|&=2C|x'-y'|^\theta+|x'-y'|+|\phi(x')+2C|x'-y'|^\theta-\phi(y')|\\
&\leq 3C|x'-y'|^\theta+|\phi(x')-\phi(y')|+2C|x'-y'|^\theta\leq6C|x-y|^\theta,
\end{align*}
so the strong regularity condition is satisfied in the case that $x$ and $y$ are both boundary points of $\Omega$.

In general, $x,y\in\overline{\Omega}\cap B(x_0,\epsilon)$, so $x=R(x',\phi(x')+\delta(x))$ and $y=R(y',\phi(y')+\delta(y))$ for some $x',y'\in B(x_0',\epsilon)$ and $\delta(x),\delta(y)>0$. 
We shall consider the following cases.
\paragraph{Case 1: $x,y\in\overline{\Omega}\cap B(x_0,\epsilon)$ and $0\leq\delta(x),\delta(y)\leq2C|x'-y'|^\theta$.}  The polygonal chain 
$\gamma:=(x,u(x',y'),v(x',y'),y)$
is contained in the polygon chain defined above for boundary points, so it admits the strong regularity condition.  

\paragraph{Case 2: $x,y\in\Omega\cap B(x_0,\epsilon)$ and $\delta(x),\delta(y)\geq2C|x'-y'|^\theta$.}
Then $\gamma:=[x,y]\subset\Omega$, and the conclusion is trivially satisfied.

\paragraph{Case 3: $x,y\in\Omega\cap B(x_0,\epsilon)$, $0\leq\delta(x)\leq2C|x'-y'|^\theta$ and $\delta(y)\geq2C|x'-y'|^\theta$.} Here for the curve $\gamma$ we choose the polygonal chain
$(x,u(x',y'),y)$ which is, by similar arguments, contained in $\Omega\cup\{x\}$. Its length is
\begin{align*}
|\gamma|&=|x-u(x',y')|+|u(x',y')-y|\leq2|x-u(x',y')|+|x-y|\\
&=4C|x'-y'|^\theta+|x-y|\leq5C|x-y|^\theta. 
\end{align*}

\paragraph{Case 4: $x,y\in\Omega\cap B(x_0,\epsilon)$, $\delta(x)\geq2C|x'-y'|^\theta$ and $0\leq\delta(y)\leq2C|x'-y'|^\theta$.}
Analogously as in case 3.

Summing up, any pair of points $x,y\in\overline{\Omega}\cap B(x_0,\epsilon)$ can be joined by a polygonal chain $\gamma$ such that $\im\gamma\setminus\{x,y\}\subset\Omega$ and $|\gamma|\leq6C|x-y|^\theta$, as desired.    
\end{proof}

\subsection*{Acknowledgements}
Much of the first author's research was conducted during a research stay in the Department of Mathematics and Statistics of the University of Helsinki, Finland. He would like to extend special thanks to Hans-Olav Tylli for his extremely kind hospitality throughout the whole stay.   

\subsection*{Funding}
The research of the first named author was partially supported by Adam Mickiewicz University, Pozna\'n, Poland, ID-UB Call No.~116 'Mobility'.


{\small
\begin{thebibliography}{99}



\bibitem{Bierstone1980} E.~Bierstone, 
\emph{Differentiable functions}, 
Bol.~Soc.~Brasil.~Mat.~{\bf 11} (1980), no.~2, 139--189.

\bibitem{Browder1962b}
F.E.~Browder,
\emph{Functional analysis and partial differential equations II},
Math.~Ann.~{\bf 145} (1961/62), 81--226.

\bibitem{Browder1962}
F.E.~Browder, 
\emph{Approximation by solutions of partial differential equations},
Amer.~J.~Math.~{\bf 84} (1962), 134--160.

\bibitem{DebKalmes2024} A.~Debrouwere, T.~Kalmes,
\emph{Quantitative Runge type approximation theorems for zero solutions of certain partial differential operators},
Adv.~Math.~{\bf 478} (2025), Paper No.~110413., 31 pp.

\bibitem{Diaz1980}
R.~Diaz,
\emph{A {R}unge theorem for solutions of the heat equation},
Proc.~Amer.~Math.~Soc.~{\bf 80} (1980), 643--646.

\bibitem{EnGaPe19}
A.~Enciso, M.A.~Garc\'{\i}a-Ferrero, D.~Peralta-Salas, 
\emph{Approximation theorems for parabolic equations and movement of local hot spots},
Duke Math.~J.~{\bf 168} (2019), 897--939.

\bibitem{EnPe12}
A.~Enciso, D.~Peralta-Salas, \emph{Knots and links in the steady solutions of the Euler equation}, Ann.~of Math.~{\bf 175} (2012), 345--367.

\bibitem{EnPe15}
A.~Enciso, D.~Peralta-Salas,
\emph{Existence of knotted vortex tubes in steady Euler flows},
Acta Math.~{\bf 214} (2015), 61–134.

\bibitem{EnPe21}
A.~Enciso, D.~Peralta-Salas,
\emph{Approximation {T}heorems for the {S}chr\"{o}dinger {E}quation and {Q}uantum {V}ortex {R}econnection},
Comm.~Math.~Phys.~{\bf 387} (2021), 1111--1149.

\bibitem{EnLuPe17}
A.~Enciso, R.~Luc\'a, D.~Peralta-Salas, \emph{Vortex reconnection in the three dimensional Navier-Stokes equations}, Adv.~Math.~{\bf 309} (2017), 452--486.

\bibitem{Frerick-Habilitation} L.~Frerick,
\emph{Extension operators for spaces of infinitely differentiable Whitney functions}, 
Habilitation thesis, Bergische Universit\"at Wuppertal, 2001.

\bibitem{Frerick07} L.~Frerick,
\emph{Extension operators for spaces of inifinite differentiable Whitney jets},
J.~reine angew.~Math.~{\bf 602} (2007), 123--153.

\bibitem{GaNe16} P.M.~Gauthier, V.~Nestoridis,
\emph{Density of polynomials in classes of functions on products of planar domains},
J.~Math.~Anal.~Appl.~{\bf 433} (2016), 282--290.

\bibitem{GT} P.M.~Gauthier, N.~Tarkhanov,
\emph{Rational approximation and universality for a quasilinear parabolic equation},
J.~Contemp.~Math.~Anal.~{\bf 43} (2008), 353--364.

\bibitem{Hor1} L.~H\"ormander, 
``The analysis of linear partial differential operators I.
Distribution theory and Fourier analysis", 
Reprint of the second (1990) edition, Classics in Mathematics, Springer-Verlag, Berlin, 2003. 

\bibitem{Hor2} L.~H\"ormander, 
``The analysis of linear partial differential operators II.
Differential operators with constant coefficients", 
Reprint of the 1983 original, Classics in Mathematics, Springer-Verlag, Berlin, 2005.

\bibitem{Jones1975}
B.F.~Jones, Jr.,
\emph{An approximation theorem of Runge type for the heat equation},
Proc.~Amer.~Math.~Soc.~{\bf 52} (1975), 289--292.

\bibitem{Kalmes2021} T.~Kalmes,
\emph{An approximation theorem of Runge type for kernels of certain non-elliptic partial differential operators},
Bull.~Sci.~math.~{\bf 170} (2021), Paper No.~103012, 26 pp.

\bibitem{Lax}
P.D.~Lax, \emph{A stability theorem for solutions of abstract differential equations, and its application to the study of the local behavior of solutions of elliptic equations},
Comm.~Pure Appl.~Math.~{\bf 9} (1956), 747--766.

\bibitem{Malgrange}
B.~Malgrange,
\emph{Existence et approximation des solution des \'equations aux deriv\'ees partielles et des \'equation de convolution},
Ann.~Inst.~Fourier (Grenoble) {\bf 6} (1955/56), 71--355.

\bibitem{Malgrange1967}
B.~Malgrange,
``Ideals of differentiable functions",
Tata Inst.~Fund.~Res.~Stud.~Math.~3, Oxford University Press, London 1967.

\bibitem{MeV} R.~Meise, D.~Vogt,
``Introduction to functional analysis",
Oxford University Press, New York 1997.

\bibitem{NeZa15} V.~Nestoridis, I.~Zadik,
\emph{Pad\'e approximants, density of rational functions in $A^\infty(\Omega)$ and the smoothness of the integral operator},
J.~Math.~Anal.~Appl.~423 (2015), 1514--1539.

\bibitem{Rainer2019} A.~Rainer,
\emph{Arc-smooth functions on closed sets}, 
Compos.~Math.~{\bf 155} (2019), no.~4, 645--680.

\bibitem{ShVi24}
A.A.~Shlapunov, P.Yu.~Vilkov, 
\emph{On Runge type theorems for solutions to strongly uniformly parabolic operators},
Sib.~\`Elektron.~Mat.~Izv. Reports {\bf 21} (2024), 383--404.

\bibitem{Treves-LCS-PDE}
F.~Tr\`eves,
``Locally convex spaces and linear partial differential equations",
Die Grundlehren der mathematischen Wissenschaften, Band 146 Springer-Verlag New York, Inc., New York 1967.

\bibitem{Whitney34}
H.~Whitney,
\emph{Analytic extensions of differentiable functions defined on closed sets},
Trans.~Amer.~Math.~Soc.~{\bf 36} (1934), 63--89.

\end{thebibliography}
}
\end{document}